\documentclass[12pt]{article}

\usepackage{etex}
\usepackage[all]{xy}
\usepackage{cite}
\usepackage{amsfonts}
\usepackage{amsthm}
\usepackage{enumerate}
\usepackage{graphicx}
\usepackage{mathrsfs}
\usepackage{bm}
\usepackage{amssymb,amsmath} 
\usepackage{makecell}
\usepackage{tikz}
\usepackage{pgf}
\usepackage{tikz}
\usetikzlibrary{patterns}
\usepackage{pgffor}
\usepackage{pgfcalendar}
\usepackage{pgfpages}
\usepackage{shuffle,yfonts}
\usepackage{mathtools}
 \allowdisplaybreaks
 
\DeclareFontFamily{U}{shuffle}{}
\DeclareFontShape{U}{shuffle}{m}{n}{ <-8>shuffle7 <8->shuffle10}{}

\newcommand{\tn}{{\tilde{n}}}

\DeclareMathOperator{\dc}{{\rm dc}}
\DeclareMathOperator{\sn}{{\rm sn}}
\DeclareMathOperator{\sd}{{\rm sd}}
\DeclareMathOperator{\nc}{{\rm nc}}
\DeclareMathOperator{\ds}{{\rm ds}}
\DeclareMathOperator{\cn}{{\rm cn}}
\DeclareMathOperator{\dn}{{\rm dn}}
\newcommand{\ga}{\alpha}

\newcommand{\gl}{\lambda}

\DeclareMathOperator\Res{{\rm Res}}

\newcommand{\bfi}{{\boldsymbol{\sl{i}}}}

 \allowdisplaybreaks
\usetikzlibrary{arrows,shapes,chains}
 \allowdisplaybreaks

\catcode`!=11
\let\!int\int \def\int{\displaystyle\!int}
\let\!lim\lim \def\lim{\displaystyle\!lim}
\let\!sum\sum \def\sum{\displaystyle\!sum}
\let\!sup\sup \def\sup{\displaystyle\!sup}
\let\!inf\inf \def\inf{\displaystyle\!inf}
\let\!cap\cap \def\cap{\displaystyle\!cap}
\let\!max\max \def\max{\displaystyle\!max}
\let\!min\min \def\min{\displaystyle\!min}
\let\!frac\frac \def\frac{\displaystyle\!frac}
\catcode`!=12

\let\oldsection\section
\renewcommand\section{\setcounter{equation}{0}\oldsection}

\allowdisplaybreaks

\def\gs{{\sigma}}
\def\ss{{s}}
\def\R{\mathbb{R}}

\def\N{\mathbb{N}}
\def\Z{\mathbb{Z}}
\def\Q{\mathbb{Q}}

\theoremstyle{plain}
\newtheorem{thm}{Theorem}[section]
\newtheorem{lem}[thm]{Lemma}
\newtheorem{conj}[thm]{Conjecture}
\newtheorem{cor}[thm]{Corollary}
\newtheorem{con}[thm]{Conjecture}
\newtheorem{pro}[thm]{Proposition}
\theoremstyle{definition}
\newtheorem{defn}{Definition}[section]
\newtheorem{re}[thm]{Remark}

\begin{document}
\title{\bf General Berndt-Type Integrals and Series Associated with Jacobi Elliptic Functions}
\author{
{Ce Xu${}^{a,}$\thanks{Email: cexu2020@ahnu.edu.cn}\quad and\quad Jianqiang Zhao${}^{b,}$\thanks{Email: zhaoj@ihes.fr}}\\[1mm]
\small a. School of Mathematics and Statistics, Anhui Normal University,\\ \small  Wuhu 241002, P.R. China\\
\small b. Department of Mathematics, The Bishop's School, La Jolla, CA 92037, USA\\[5mm]
\normalsize \emph{Dedicated to Professor Bruce C. Berndt on the occasion of his 85th birthday}
}

\date{}
\maketitle

\noindent{\bf Abstract.}  In this paper, we prove two structural theorems on the general Berndt-type integrals with the denominator having arbitrary positive degrees by contour integrations involving hyperbolic and trigonometric functions, and hyperbolic sums associated with Jacobi elliptic functions.  We first establish explicit relations between these integrals and four classes of hyperbolic sums. Then, using our previous results on hyperbolic series and applying the matrix method from linear algebra, we compute explicitly several general hyperbolic sums and their higher derivatives. These enable us to express two families of general Berndt-type integrals as polynomials in $\Gamma^4(1/4)$ and $\pi^{-1}$ with rational coefficients, where $\Gamma$ is the Euler gamma function. At the end of the paper, we provide some conjectures of general Berndt-type integrals.

\medskip



\section{Introduction}
For nonnegative integer $a$ and positive integer $b$, the \emph{Berndt-type integrals of order $b$} are of the form
\begin{align}\label{BTI-definition-1}
\int_0^\infty \frac{x^{a}dx}{(\cos x\pm\cosh x)^b},
\end{align}
where $a\geq 0$ and $b\geq 1$ if the denominator has ``$+$' sign, and $a\geq 2b$ otherwise.

The study of Berndt-type integrals has a long history, first as a problem submitted by Ramanujan \cite[pp. 325-326]{Rama1916}
over a century ago to the \emph{Indian Journal of Pure and Applied Mathematics}:
\begin{equation*}
\int_0^\infty \frac{\sin(nx)}{x(\cos x+\cosh x)}dx=\frac{\pi}{4} \quad \text{(for any odd integer $n$)}.
\end{equation*}
Wilkinson \cite{W1916} provided a proof four years later.
At the International Conference on Orthogonal Polynomials and $q$-Series, which was held in Orlando in May 2015 in celebration of the 70th birthday of Professor Mourad Ismail, Dennis Stanton gave a plenary talk titled ``A small slice of Mourad's work". One of the topics in that talk was about ``the mystery integral of Mourad Ismail" \cite{K2017}:
\begin{align}\label{inte-Mourad-Ismail}
\int_{-\infty}^\infty \frac{dx}{\cos(K \sqrt{x})+\cosh(K' \sqrt{x})}=2,
\end{align}
where $K\equiv K(x)$ denotes the complete elliptic integral of the first kind, see \eqref{defn-elliptic-first}. This curious integral \eqref{inte-Mourad-Ismail} first appeared in \cite{Ismail1998} by Ismail and Valent. Berndt \cite{Berndt2016} provided a direct evaluation of \eqref{inte-Mourad-Ismail} and other similar integrals of this type by using residue computations, the Fourier series expansions, and the Maclaurin series expansions of Jacobi elliptic functions. In particular, he proved that some Berndt-type integrals of order one can be evaluated by special values of the Gamma function. For example,
\begin{alignat*}{4}
\int_0^\infty \frac{x^3}{\cos x-\cosh x}=&\, -\frac{\Gamma^8(1/4)}{256\pi^2},\quad
& \int_0^\infty \frac{x^7}{\cos x-\cosh x}=&\,\frac{9\Gamma^{16}(1/4)}{2^{13}\pi^4},\\
\int_0^\infty \frac{x^5}{\cos x+\cosh x}=&\, \frac{3\pi^3\Gamma^6(1/4)}{256\Gamma^6(3/4)},\quad
& \int_0^\infty \frac{x^9}{\cos x+\cosh x}=&\,\frac{3^3\cdot 7\pi^5\Gamma^{10}(1/4)}{2^{12}\Gamma^{10}(3/4)}.
\end{alignat*}
It turns out that Berndt's results and methods can be generalized and organized further.  In our previous paper \cite{XZ2023}, by extending an argument as used in the proof of the main theorem of \cite{Berndt2016}, we proved that the following explicit evaluations of two Berndt-type integrals of order two (cf. \cite[Thm. 1.3]{XZ2023}):
\begin{align*}
&\int_0^\infty \frac{x^{4p+1}}{(\cos x-\cosh x)^2}dx\in \Q\frac{\Gamma^{8p}(1/4)}{\pi^{2p}}+\Q\frac{\Gamma^{8p+8}(1/4)}{\pi^{2p+4}}\quad (p\in \N),\\
&\int_0^\infty \frac{x^{4p+1}}{(\cos x+\cosh x)^2}dx\in \Q\frac{\Gamma^{8p}(1/4)}{\pi^{2p}}+\Q\frac{\Gamma^{8p+8}(1/4)}{\pi^{2p+4}}\quad (p\in \N).
\end{align*}
Similarly, we also proved the following explicit evaluations of two Berndt-type integrals of order one ($\N_0:=\N \cup \{0\}$)
\begin{align*}
&\int_0^\infty \frac{x^{4p-1}}{\cos x-\cosh x}dx\in \Q \frac{\Gamma^{8p}(1/4)}{\pi^{2p}}\quad (p\in \N),\\
&\int_0^\infty \frac{x^{4p+1}}{\cos x+\cosh x}dx\in \Q \frac{\Gamma^{8p+4}(1/4)}{\pi^{2p+1}}\quad (p\in \N_0).
\end{align*}
Further, we proved the following explicit evaluations of Berndt-type integrals of order three (see \cite[Theorem 1.1]{RXZ2023}):
putting $\Gamma:=\Gamma(1/4)$, for any positive integer $p$, we have
\begin{align*}
	\int_{0}^{\infty}\frac{x^{4p+1}\, dx}{(\cos x+\cosh x)^3}\in &\Q\frac{\Gamma^{8p-4}}{\pi^{2p-1}}+\Q\frac{\Gamma^{8p+4}}{\pi^{2p+3}}+\Q\frac{\Gamma^{8p+4}}{\pi^{2p+2}}+\Q\frac{\Gamma^{8p+4}}{\pi^{2p+1}}+\Q\frac{\Gamma^{8p+12}}{\pi^{2p+7}},
	\\\int_{0}^{\infty}\frac{x^{4p-1}\, dx}{\left(\cos x-\cosh x\right)^3}\in&\Q\frac{\Gamma^{8p-8}}{\pi^{2p-2}}+\Q\frac{\Gamma^{8p}}{\pi^{2p+2}}+\Q\frac{\Gamma^{8p}}{\pi^{2p+1}}+\Q\frac{\Gamma^{8p}}{\pi^{2p}}+\Q\frac{\Gamma^{8p+8}}{\pi^{2p+6}}\quad (p>1).
\end{align*}

In this paper, we use the contour integrals and the series associated with Jacobi
elliptic functions to establish two structural theorems on the general Berndt-type integrals
\eqref{BTI-definition-1} with the denominator having arbitrary positive degrees.
We will prove the following evaluations (see Theorems \ref{thm:Plus} and \ref{thm:Minus}, respectively).

\begin{thm} \label{thm:Plus-0}
Let $X=\Gamma^4(1/4)$ and $Y=\pi^{-1}$. For all integers $m\geq 1$ and $p\geq [m/2]$, the Berndt-type integrals
\begin{equation*}
\int_0^\infty \frac{x^{4p+1} dx}{(\cos x+\cosh x)^{m}} \in   \Q[X,Y]
\end{equation*}
where the degrees of $X$ have the same parity as that of $m$
and are between $2p-m+2$ and $2p+m$, inclusive, while
the degrees of $Y$ are between $2p-m+2$ and $2p+3m-2$, inclusive.
\end{thm}

\begin{thm} \label{thm:Minus-0}
Let $X=\Gamma^4(1/4)$ and $Y=\pi^{-1}$. For all positive integers $a$ and $m$ satisfying $0<a-2m\equiv 1 \pmod{4}$
\begin{equation*}
\int_0^\infty \frac{x^{a} dx}{(\cos x-\cosh x)^{m}} \in \Q[X,Y]
\end{equation*}
where the degrees of $X$ are always even and are between $(a+3)/2-m$ and $(a-1)/2+m$, inclusive,
while the degrees of $Y$ are between $(a+3)/2-m$ and $(a-5)/2+3m$, inclusive.
\end{thm}

We remark that all the degree bounds in Theorem~\ref{thm:Plus-0} and Theorem~\ref{thm:Minus-0} are optimal. We record six examples to illustrate Theorem~\ref{thm:Plus-0} and Theorem~\ref{thm:Minus-0} (for the first four examples, see \cite[Examples 4.2 and 4.5]{RXZ2023} and \cite[Examples 6.4 and 8.13]{XuZhao-2022}).
\begin{align*}
&\int_0^\infty \frac{x^{9}}{(\cos x+\cosh x)^2}dx=-\frac{189\Gamma^{16}(1/4)}{5\cdot 2^{15}\pi^{4}}+\frac{9\Gamma^{24}(1/4)}{2^{21}\pi^{8}},\\
&\int_0^\infty \frac{x^{9}}{(\cos x-\cosh x)^2}dx=\frac{27\Gamma^{16}(1/4)}{5\cdot 2^{12}\pi^{4}}-\frac{\Gamma^{24}(1/4)}{2^{18}\pi^{8}},\\
&\int_{0}^{\infty}\frac{x^{11}\, dx}{(\cos x-\cosh x)^3}=-\frac{4455\Gamma^{16}}{2^{15}\pi^4}-\frac{189\Gamma^{24}}{2^{20}\pi^6}+\frac{297\Gamma^{24}}{2^{18}\pi^7}-\frac{935\Gamma^{24}}{2^{20}\pi^8}
-\frac{195\Gamma^{32}}{2^{27}\pi^{12}},\\
&\int_{0}^{\infty}\frac{x^{13}\, dx}{(\cos x+\cosh x)^3}=\frac{405405\Gamma^{20}}{2^{20}\pi^5}+\frac{68607\Gamma^{28}}
{2^{27}\pi^7}-\frac{107757\Gamma^{28}}{2^{25}\pi^8}+\frac{84591\Gamma^{28}}{2^{25}\pi^9}+\frac{17679\Gamma^{36}}{2^{32}\pi^{13}},\\
&\int_0^\infty \frac{x^{33}}{(\cos x-\cosh x)^2}dx=\frac{55168390953244107 \Gamma^{64}}{85\cdot 2^{36} \pi ^{16}}-\frac{135515509591329 \Gamma^{72}}{2^{42} \pi ^{20}},\\
&\int_0^\infty \frac{x^{33}}{(\cos x+\cosh x)^2}dx=-\frac{1807702666364949654069 \Gamma^{64}}{85\cdot 2^{51}  \pi ^{16}}+\frac{4440707733798260001 \Gamma^{72}}{2^{57} \pi ^{20}},
\end{align*}
where $\Gamma:=\Gamma(1/4)$.

\section{Notation and Some Preliminary Results}
The Gaussian or ordinary hypergeometric function ${_2}F_1(a,b;c;x)$ is defined for $|x|<1$ by the power series
\begin{equation*}
{_2}F_1(a,b;c;x)=\sum\limits_{n=0}^\infty \frac{(a)_n(b)_n}{(c)_n} \frac {x^n}{n!}\quad (a,b,c\in\mathbb{C}).
\end{equation*}
Here $(a)_n$ denotes the \emph{ascending Pochhammer symbol} defined by
\begin{equation*}
(a)_n:=\frac {\Gamma(a+n)}{\Gamma(a)}=a(a+1)\cdots(a+n-1) \quad {\rm and}\quad (a)_0:=1,
\end{equation*}
where $a\in \mathbb{C}$ is any complex number and $n\in\N_0$ is a nonnegative integer. Here $\Gamma(s)$ denotes
the Gamma function. When $\Re(s) > 0$
\[\Gamma(s) := \int\limits_0^\infty  {{e^{ - t}}{t^{s - 1}}dt}.\]

The complete elliptic integral of the first kind is defined by (Whittaker and Watson \cite{WW1966})
\begin{align}
K:=K(x):=K(k^2):=\int\limits_{0}^{\pi/2}\frac {d\varphi}{\sqrt{1-k^2\sin^2\varphi}}=\frac {\pi}{2} {_2}F_{1}\left(\frac {1}{2},\frac {1}{2};1;k^2\right).\label{defn-elliptic-first}
\end{align}
Here $x=k^2$ and $k\ (0<k<1)$ is the modulus of $K$. The complementary modulus $k'$ is defined by $k'=\sqrt {1-k^2}$. Furthermore,
\begin{equation*}
K':=K(k'^2)=\int\limits_{0}^{\pi/2}\frac {d\varphi}{\sqrt{1-k'^2\sin^2\varphi}}=\frac {\pi}{2} {_2}F_{1}\left(\frac {1}{2},\frac {1}{2};1;1-k^2\right).
\end{equation*}
Similarly, the complete elliptic integral of the second kind is denoted by (Whittaker and Watson \cite{WW1966})
\begin{equation*}
E:=E(x):=E(k^2):=\int\limits_{0}^{\pi/2} \sqrt{1-k^2\sin^2\varphi}d\varphi=\frac {\pi}{2} {_2}F_{1}\left(-\frac {1}{2},\frac {1}{2};1;k^2\right).
\end{equation*}

In order to better state our main results, we shall henceforth adopt the notations of Ramanujan (see Berndt's book \cite{B1991}). Let
\begin{align}\label{den-z-diff}
&x:=k^2,\quad y:=y(x):=\pi \frac {K'}{K}, \quad q:=q(x):=e^{-y},\nonumber \\
&z:=z(x):=\frac {2}{\pi}K,\quad z':=\frac{dz}{dx},\quad z":=\frac{d^2z}{dx^2},\quad z^{(n)}:=\frac{d^nz}{dx^n}.
\end{align}
Using the identity $(a)_{n+1}=a(a+1)_n$, it is easily shown that
\[\frac {d}{dx}{_2}F_1(a,b;c;x)=\frac{ab}{c}{_2}F_1(a+1,b+1;c+1;x)\]
and more generally,
\begin{equation*}
\frac {d^n}{dx^n}{_2}F_1(a,b;c;x)=\frac{(a)_n(b)_n}{(c)_n}{_2}F_1(a+n,b+n;c+n;x)\quad(n\in\N_0).\label{1.7}
\end{equation*}
Then,
\begin{equation*}
\frac {d^nz}{dx^n}=\frac{\left(1/2\right)^2_n}{n!}{_2}F_1\left(\frac{1}{2}+n,\frac{1}{2}+n;1+n;x\right).
\end{equation*}
Applying the identity (see Andrews-Askey-Roy's book \cite[Chapter 3, Section 3.5, Theorem 3.5.4(i)]{A2000})
\begin{equation*}
{_2}F_1\left(a,b;\frac {a+b+1}{2};\frac 1{2}\right)=\frac{\Gamma\left(\frac{1}{2}\right)\Gamma\left(\frac{a+b+1}{2}\right)}{\Gamma\left(\frac{a+1}{2}\right)\Gamma\left(\frac{b+1}{2}\right)},
\end{equation*}
letting $a=b=1/2+n$ and noting the fact that $\Gamma(1/2)=\sqrt{\pi}$, we can find by an elementary calculation that 
\begin{align}\label{equ:znDer}
\left.\frac {d^nz}{dx^n}\right|_{x=1/2}=\frac{(1/2)^2_n\sqrt{\pi}}{\Gamma^2\left(\frac{n}{2}+\frac {3}{4}\right)}.
\end{align}
Clearly, for any $m\in \N$, we have 
\begin{align*}
\left.\frac {d^{2m}z}{dx^{2m}}\right|_{x=1/2}\in \mathbb{Q} \frac{\Gamma^2(1/4)}{\pi^{3/2}}\quad \text{and} \quad \left.\frac {d^{2m-1}z}{dx^{2m-1}}\right|_{x=1/2}\in \mathbb{Q}\frac{\sqrt{\pi}}{\Gamma^2(1/4)}.
\end{align*}
In particular, taking $x=1/2,n=0,1,2,3$ in \eqref{den-z-diff} and applying \eqref{equ:znDer}, we obtain that
\begin{align*}
&y\Big(\frac 1{2}\Big)=\pi,\quad  z\Big(\frac 1{2}\Big)=\frac{\Gamma^2(1/4)}{2\pi^{3/2}},\quad
z'\Big(\frac 1{2}\Big)=\frac {4\sqrt{\pi}}{\Gamma^2(1/4)},\\
&z''\Big(\frac 1{2}\Big)=\frac{\Gamma^2(1/4)}{2\pi^{3/2}},\quad
z^{(3)}\Big(\frac{1}{2}\Big)=\frac{36\sqrt{\pi}}{\Gamma^2(1/4)},
\end{align*}
where we have used the two classical relations
\begin{equation*}
\Gamma(s+1)=s\Gamma(s)
\end{equation*}
and
\begin{equation*}
\Gamma(s)\Gamma(1-s)=\frac{\pi}{\sin(\pi s)}\quad (s\in \mathbb{C}\setminus \mathbb{Z}).
\end{equation*}

\section{Berndt-Type Integrals via Hyperbolic Series}
To save space, throughout the rest of the paper we will put
\begin{equation*}
\tn=\frac{2n-1}2.
\end{equation*}
We need the following lemma.
\begin{lem}\emph{(cf.\cite{XZ2023})} Let $n$ be an integer. Then we have
\begin{align}
&\frac{(-1)^n}
{{\cosh \left(\frac{1+i}{2}z \right)}}
\buildrel{z\to\tn\pi(1+i)}\over{=\joinrel=\joinrel=\joinrel=\joinrel=\joinrel=} 2i \left\{\frac1{1+i}\cdot\frac{1}{z-\tn\pi(1+i)}\atop +\sum_{k=1}^\infty (-1)^k \frac{{\bar \zeta}(2k)}{\pi^{2k}}\left(\frac{1+i}{2}\right)^{2k-1}\left(z-\tn\pi(1+i)\right)^{2k-1} \right\},\label{eq-cosh-1}\\
&\frac{(-1)^n}
{{\sinh \left(\frac{1+i}{2}z \right)}}
\buildrel{z \to n\pi (1+i)}\over{=\joinrel=\joinrel=\joinrel=\joinrel=\joinrel=}   2\left\{\frac1{1+i}\cdot\frac{1}{z-n\pi(1+i)}\atop +\sum_{k=1}^\infty (-1)^k \frac{{\bar \zeta}(2k)}{\pi^{2k}}\left(\frac{1+i}{2}\right)^{2k-1}\left(z-n\pi(1+i)\right)^{2k-1} \right\}\label{eq-sinh-1}.
\end{align}
\end{lem}
Here  $\zeta(s)$ and $\bar{\zeta}(s)$ are \emph{Riemann zeta function} and \emph{alternating Riemann zeta function}, respectively, defined by
\begin{align*}
\zeta(s):=\sum_{n=1}^{\infty}\frac{1}{n^s}\quad\left(\Re\left(s\right)>1\right)\quad \text{and}\quad \bar{\zeta}(s):=\sum_{n=1}^{\infty}\frac{(-1)^{n-1}}{n^s}\quad\left(\Re\left(s\right)>0\right).
\end{align*}
For even $s=2m\ (m\in \N)$ Euler proved the famous formula
\begin{align}\label{equ=even-zeta}
\zeta(2m) = -\frac12\frac{B_{2m}}{(2m)!}(2\pi i)^{2m},
\end{align}
where $B_{2m}\in \Q$ ($B_0=1,B_1=-\frac1{2},B_2=\frac1{6},B_4=\frac1{30},\ldots$) are \emph{Bernoulli numbers}  defined by the generating function
\begin{equation*}
 \frac{x}{e^x-1}=\sum_{n=0}^\infty \frac{B_n}{n!}x^n.
\end{equation*}

\begin{pro} Set $z_n:=\tn\pi(1+i)$ and $y_n:=n\pi (1+i)$. For positive integers $m$ and $n$, we have
\begin{align}\label{equ-power-m-hybolic-function}
\frac{(-1)^{mn}}
{{\cosh^m \left(\frac{1+i}{2}z \right)}}
\buildrel{z\to z_n}\over{=\joinrel=\joinrel=} (i+1)^m \sum_{l=0}^\infty i^l \ga_{l,m}  (z-z_n)^{2l-m}
\end{align}
and
\begin{align}\label{equ-power-m-hybolic-sinh-function}
\frac{(-1)^{mn}}
{{\sinh^m \left(\frac{1+i}{2}z \right)}}
\buildrel{z\to y_n}\over{=\joinrel=\joinrel=} (1-i)^m \sum_{l=0}^\infty i^l \ga_{l,m}  (z-y_n)^{2l-m},
\end{align}
where
\begin{align}\label{defn-Coffe-C}
\ga_{l,m}:=(-1)^m 2^l \sum_{k_1+\cdots+k_m=l,\atop k_j\geq 0, \forall j} \prod\limits_{j=1}^m \frac{(1-2^{1-2k_j})B_{2k_j}}{(2k_j)!}\in \Q.
\end{align}
In particular, we have $\ga_{0,m}=1$ and $\ga_{1,m}=-\frac{m}{12}$.
\end{pro}
\begin{proof}
We only prove that \eqref{equ-power-m-hybolic-function} holds, the proof of \eqref{equ-power-m-hybolic-sinh-function} is completely similar to that of \eqref{equ-power-m-hybolic-function}.
We note that the \eqref{eq-cosh-1} can be rewritten as
\begin{align*}
&\frac{(-1)^n}
{{\cosh \left(\frac{1+i}{2}z \right)}}
\buildrel{z\to z_n}\over{=\joinrel=\joinrel=} 2i \sum_{k=0}^\infty (-1)^k \frac{{\bar \zeta}(2k)}{\pi^{2k}}\left(\frac{1+i}{2}\right)^{2k-1}\left(z-z_n\right)^{2k-1},
\end{align*}
where ${\bar \zeta}(0):=\frac1{2}$. Hence, by a direct calculation, one obtains
\begin{align*}
&\frac{(-1)^{mn}}
{{\cosh^m\left(\frac{1+i}{2}z \right)}}
\buildrel{z\to z_n}\over{=\joinrel=\joinrel=} (2i)^m \sum_{l=0}^\infty (-1)^l \left(\frac{1+i}{2}\right)^{2l-m} \left\{\sum_{k_1+\cdots+k_m=l,\atop k_j\geq 0, \forall j} \frac{{\bar \zeta}(2k_1)\cdots {\bar \zeta}(2k_m)}{\pi^{2k_1+\cdots+2k_m}} \right\}  \left(z-z_n\right)^{2l-m}.
\end{align*}
Noting the fact that for $\Re(s)>1$,
\begin{align*}
{\bar \zeta}(s)=(1-2^{1-s})\zeta(s)\quad \text{and}\quad \zeta(0):=-\frac1{2},
\end{align*}
and applying \eqref{equ=even-zeta} yields the desired evaluation \eqref{equ-power-m-hybolic-function} with an elementary calculation.
\end{proof}

\begin{thm}\label{main-thm-one-plus}
For any positive integer $m$ and real $a>0$, we have
\begin{align*}
&\frac{2^{m-2}(1-i^{a+1})}{(1+i)^{a-1}}\int_0^\infty \frac{x^a dx}{(\cos x+\cosh x)^m}\nonumber\\
=&\, \sum_{l=0}^{[(m-1)/2]}\sum_{j=0}^{m-1-2l}\frac{(-1)^{l+1}2^l \ga_{l,m}}{(m-1-2l-j)!}\binom{a}{j}\pi^{a+1-j}
\lim_{y\to \pi}\frac{d^{m-1-2l-j}}{dy^{m-1-2l-j}}\left\{\sum_{n=1}^\infty (-1)^{mn} \frac{\tn^{a+1+2l-m}}{\cosh^m (\tn y)}\right\},
\end{align*}
where $\ga_{l,m}$ is given by \eqref{defn-Coffe-C}.
\end{thm}
\begin{proof}
Let $z=x+iy$ for $x,y\in \R$. Consider
\begin{align*}
\lim\limits_{R\to\infty}\int_{C_R}\frac{z^a\, dz}{(\cos z+\cosh z)^m}=\lim\limits_{R\to\infty}\int_{C_R}F\left(z\right)\, dz,
\end{align*}
where $C_R$ denotes the positively oriented quarter-circular contour consisting of the interval $[0,R]$, the quarter-circle $\Gamma_R$ with $|z|=R$ and $0\leq \arg z\leq \pi/2$, and $[iR,0]$ (i.e., the line segment from $i R$ to $0$ on the imaginary axis).

Clearly, there exist poles of order $m$ when
\begin{align*}
\cos z+\cosh z=2\cosh\left\{\frac1{2}(z+iz)\right\}\cosh\left\{\frac1{2}(z-iz)\right\}=0.
\end{align*}
Hence, there exist one set of poles at
\begin{align*}
z_n:=\frac{(2n-1)\pi i}{1+i}=\tn (1+i)\pi ,\quad n\in \Z,
\end{align*}
and
\begin{align*}
s_n:=\frac{(2n-1)\pi i}{1-i}=\tn(i-1)\pi ,\quad n\in \Z.
\end{align*}
The only poles lying inside $C_R$ are $z_n,\ n\geq 1,\ |z_n|<R$. Using \eqref{equ-power-m-hybolic-function}, if $z\rightarrow z_n$ then
\begin{align*}
F\left(z\right)\buildrel{z\rightarrow z_n}\over{=\joinrel=}\frac{(-1)^{mn}(i+1)^m}{2^m} \frac{z^a}{\cosh^m\left(\frac{1-i}{2}z\right)} \sum_{l=0}^\infty i^l \ga_{l,m}  (z-z_n)^{2l-m}.
\end{align*}
The residue $\underset{z=z_n}\Res \{F(z)\}$ at such a pole $z_n$ is give by
\begin{align*}
\underset{z=z_n}\Res \{F(z)\}&=\frac{1}{(m-1)!}\lim_{z\rightarrow z_n} \frac{d^{m-1}}{dz^{m-1}} \left\{(z-z_n)^m F(z)\right\}\\
&=\frac{(-1)^{mn}(i+1)^m}{(m-1)!2^m}  \sum_{l=0}^\infty i^l \ga_{l,m} \lim_{z\rightarrow z_n} \frac{d^{m-1}}{dz^{m-1}} \left\{(z-z_n)^{2l}\frac{z^a}{\cosh^m\left(\frac{1-i}{2}z\right)} \right\}\\
&=\frac{(-1)^{mn}(i+1)^m}{(m-1)!2^m}  \sum_{l=0}^{[(m-1)/2]} i^l \ga_{l,m} \binom{m-1}{2l}(2l)! \lim_{z\rightarrow z_n}\frac{d^{m-1-2l}}{dz^{m-1-2l}} \left\{\frac{z^a}{\cosh^m\left(\frac{1-i}{2}z\right)} \right\}\\
&=\frac{(-1)^{mn}(i+1)^m}{2^m}  \sum_{l=0}^{[(m-1)/2]} \frac{i^l  \ga_{l,m}}{(m-1-2l)!}\lim_{z\rightarrow z_n}\sum_{j=0}^{m-1-2l} \binom{m-1-2l}{j} \\&\quad\quad\quad\quad\quad\times a(a-1)\cdots(a-j+1)z^{a-j}\frac{d^{m-1-2l-j}}{dz^{m-1-2l-j}}  \left\{\frac{1}{\cosh^m\left(\frac{1-i}{2}z\right)} \right\}\\
&=\frac{(-1)^{mn}(i+1)^m}{2^m} \sum_{l=0}^{[(m-1)/2]}\sum_{j=0}^{m-1-2l}  \frac{i^l  \ga_{l,m}}{(m-1-2l-j)!}\binom{a}{j}z_n^{a-j} \\&\quad\quad\quad\quad\quad\quad\quad\quad\quad\times \lim_{z\rightarrow z_n}\frac{d^{m-1-2l-j}}{dz^{m-1-2l-j}}  \left\{\frac{1}{\cosh^m\left(\frac{1-i}{2}z\right)} \right\}.
\end{align*}
Noting that
\begin{align*}
\lim_{z\rightarrow z_n}\frac{d^{p}}{dz^{p}}  \left\{\frac{1}{\cosh^m\left(\frac{1-i}{2}z\right)} \right\}=\left(\frac{1}{(1+i)\tn}\right)^p\lim_{y\rightarrow \pi}\frac{d^{p}}{dy^{p}}\left\{\frac{1}{\cosh^m(\tn y)} \right\},
\end{align*}
where we have used $z=\frac{2\tn}{1-i}y$. Hence, we obtain
\begin{align*}
\underset{z=z_n}\Res \{F(z)\}&=\frac{(1+i)^{a+1}(-1)^{mn}}{2^{m}}\sum_{l=0}^{[(m-1)/2]}\sum_{j=0}^{m-1-2l}\frac{(-1)^{l}2^l \ga_{l,m}}{(m-1-2l-j)!}\binom{a}{j}\pi^{a-j}\nonumber\\&\quad\quad\quad\quad\quad\quad\quad\quad\quad\quad\quad\quad\times\lim_{y\rightarrow \pi}\frac{d^{m-1-2l-j}}{dy^{m-1-2l-j}}\left\{\frac{\tn^{a+1+2l-m}}{\cosh^m(\tn y)} \right\}.
\end{align*}
Note that $z=i y$ for the integral over $[iR,0]$. Hence,
\begin{align*}
\int_{iR}^{0}\frac{z^a\, dz}{(\cos z+\cosh z)^m}=-i^{a+1}\int_{0}^{R}\frac{y^a\, dy}{\left(\cos y+\cosh y\right)^m}.
\end{align*}
It is easy to show that, as $R\rightarrow \infty$
\begin{align*}
\lim\limits_{R\to\infty}\int_{\Gamma_R}\frac{z^a\, dz}{(\cos z+\cosh z)^m}=o\left(1\right).
\end{align*}
Applying the residue theorem, letting $R\rightarrow \infty$, we conclude that
\begin{align*}
\left(1-i^{a+1}\right)\int_{0}^{\infty}\frac{x^a\, dx}{\left(\cos x+\cosh x\right)^m}=2\pi i\sum_{n=1}^{\infty}\underset{z=z_n}\Res \{F(z)\}.
\end{align*}
This completes the proof of the theorem.
\end{proof}

Letting $m=1,2$ in Theorem \ref{main-thm-one-plus}, we can get the following corollary.
\begin{cor}
\emph{(cf. \cite{Berndt2016,XZ2023})} For $a>0$,
\begin{align*}
\frac{(1-i^{a+1})}{(1+i)^{a-1}}\int_0^\infty \frac{x^a dx}{\cos x+\cosh x}=&\,
    2\pi^{a+1} \sum_{n=1}^\infty (-1)^{n-1} \frac{\tn^a}{\cosh(\tn \pi)},\\
\frac{(1-i^{a+1})}{(1+i)^{a-1}}\int_0^\infty \frac{x^a dx}{(\cos x+\cosh x)^2}
=&\, 2\pi^{a+1}\sum_{n=1}^\infty \frac{\tn^a\sinh(\tn \pi)}{\cosh^3(\tn \pi)}
  -a\pi^a\sum_{n=1}^\infty \frac{\tn^{a-1}}{\cosh^2(\tn \pi)}.
\end{align*}
\end{cor}

\begin{thm}\label{main-thm-one-minus}
For any positive integer $m$ and real $a\geq 2m$, we have
\begin{align*}
&\frac{2^{m-2}(1-(-1)^m i^{a+1})}{i^{m}(1+i)^{a-1}}\int_0^\infty \frac{x^a dx}{(\cos x-\cosh x)^m}\nonumber\\
=&\, \sum_{l=0}^{[(m-1)/2]}\sum_{j=0}^{m-1-2l}\frac{(-1)^{l+1}2^l \ga_{l,m}}{(m-1-2l-j)!}\binom{a}{j}\pi^{a+1-j}
\lim_{y\to \pi}\frac{d^{m-1-2l-j}}{dy^{m-1-2l-j}}\left\{\sum_{n=1}^\infty (-1)^{mn} \frac{n^{a+1+2l-m}}{\sinh^m (n y)}\right\},
\end{align*}
where $\ga_{l,m}$ is given by \eqref{defn-Coffe-C}.
\end{thm}
\begin{proof}
Consider
\begin{align*}
\lim_{R\rightarrow\infty }\int_{C_R} \frac{z^{a}dz}{(\cos z-\cosh z)^m}=\lim_{R\rightarrow\infty }\int_{C_R} G(z)dz,
\end{align*}
where $C_R$ denotes the positively oriented quarter-circular contour consisting of the interval $[0,R]$; the quarter-circle $\Gamma_R$ with $|z|=R$ and $0\leq \arg\leq \pi/2$; and the segment $[i R,0]$ on the imaginary axis. Clearly, there exist poles of order $2$ in the present case when
\begin{align*}
\cos z-\cosh z=2\sinh\left\{\frac1{2}(z+iz)\right\}\sinh\left\{\frac1{2}(iz-z)\right\}=0.
\end{align*}
Hence, there exist one set of poles at $t_n=n\pi(1-i)$ and $y_n=n\pi(1+i)$, where $n\in\N$. Those lying on the interior of $C_R$ are $y_n=n\pi (1+i),n\geq 1,|y_n|<R$. Hence, by \eqref{equ-power-m-hybolic-sinh-function}, we have
\begin{align*}
&G(z)\mathop  = \limits^{z \to y_n} \frac{(1-i)^m(-1)^{mn}}{2^m}\frac{z^a}{{\sinh^m \left(\frac{i-1}{2}z \right)}}
\sum_{l=0}^\infty i^l \ga_{l,m}  (z-y_n)^{2l-m}.
\end{align*}
Further, the residue is
\begin{align*}
\underset{z=y_n}\Res  \{G(z)\}&=\frac1{(m-1)!}\lim_{z\rightarrow y_n}\frac{d^{m-1}}{dz^{m-1}} \left\{(z-y_n)^m \frac{z^a}{(\cos z-\cosh z)^m}\right\}\nonumber\\
&=\frac{(-1)^{mn}(1-i)^m}{2^m} \sum_{l=0}^{[(m-1)/2]}\sum_{j=0}^{m-1-2l}  \frac{i^l \ga_{l,m}}{(m-1-2l-j)!}\binom{a}{j}y_n^{a-j}\nonumber \\&\quad\quad\quad\quad\quad\quad\quad\quad\quad\times \lim_{z\rightarrow y_n}\frac{d^{m-1-2l-j}}{dz^{m-1-2l-j}}  \left\{\frac{1}{\sinh^m\left(\frac{i-1}{2}z\right)} \right\}.
\end{align*}
Also,
\begin{align*}
\lim_{z\rightarrow y_n}\frac{d^{p}}{dz^{p}}  \left\{\frac{1}{\sinh^m\left(\frac{i-1}{2}z\right)} \right\}=(-1)^m\left(\frac{1}{(1+i) n}\right)^p\lim_{y\rightarrow \pi}\frac{d^{p}}{dy^{p}}\left\{\frac{1}{\sinh^m(\tn y)} \right\},
\end{align*}
where we have used $z=\frac{2 n}{1-i}y$. Hence, one obtains
\begin{align}\label{residue-case-g-1}
\underset{z=y_n}\Res  \{G(z)\}&=\frac{i^m(1+i)^{a+1}(-1)^{mn}}{2^{m}}\sum_{l=0}^{[(m-1)/2]}\sum_{j=0}^{m-1-2l}\frac{(-1)^{l}2^l \ga_{l,m}}{(m-1-2l-j)!}\binom{a}{j}\pi^{a-j}\nonumber\\&\quad\quad\quad\quad\quad\quad\quad\quad\quad\quad\quad\quad\times\lim_{y\rightarrow \pi}\frac{d^{m-1-2l-j}}{dy^{m-1-2l-j}}\left\{\frac{n^{a+1+2l-m}}{\sinh^m(n y)} \right\}.
\end{align}
The integral over $\Gamma_R$ tends to $0$ as $R\rightarrow\infty$. Hence, by \eqref{residue-case-g-1} and the residue theorem
\begin{align*}
(1-(-1)^mi^{a+1})\int_{0}^\infty \frac{x^adx}{(\cos x-\cosh x)^m}=2\pi i \sum_{n=1}^\infty \underset{z=y_n}\Res  \{G(z)\}.
\end{align*}
Thus, we have finished the proof of the theorem.
\end{proof}

Letting $m=1,2$ in Theorem \ref{main-thm-one-minus}, we can get the following corollary.
\begin{cor}
\emph{(cf. \cite{Berndt2016,XZ2023})} For $a\geq 2$,
\begin{align*}
\frac{1+i^{a+1}}{i(1+i)^{a-1}}\int_0^\infty \frac{x^a dx}{\cos x-\cosh x}
=2\pi^{a+1} \sum_{n=1}^\infty (-1)^{n+1} \frac{n^a}{\sinh(n \pi)},
\end{align*}
and for $a\geq 4$,
\begin{align*}
\frac{1-i^{a+1}}{(1+i)^{a-1}}\int_0^\infty \frac{x^a dx}{(\cos x-\cosh x)^2}&=a\pi^a\sum_{n=1}^\infty \frac{n^{a-1}}{\sinh^2(n\pi)}-2\pi^{a+1}\sum_{n=1}^\infty \frac{n^a\cosh(n \pi)}{\sinh^3(n \pi)}.
\end{align*}
\end{cor}

\section{Explicit evaluations of Hyperbolic Sums}

In this section, we will establish explicit evaluations of several hyperbolic sums associated with Jacobi elliptic functions.

Recall that the Jacobi elliptic function $\sn (u)=\sn (u,k)$ is defined via the inversion of the elliptic integral
\begin{align*}
u=\int_0^\varphi \frac{dt}{\sqrt{1-k^2\sin^2 t}}\quad (0<k^2<1),
\end{align*}
namely, $\sn (u):=\sin \varphi$. As before, we refer $k={\rm mod}\, u$ as the elliptic modulus. We also write $\varphi={\rm am}(u,k)={\rm am}(u)$ and call it the Jacobi amplitude. Then the Jacobi elliptic functions $\cn u$ and $\dn u$ may be defined by
\begin{align*}
&\cn (u):=\sqrt{1-\sn^2 (u)}\quad\text{and} \quad  \dn (u):=\sqrt{1-k^2\sn^2(u)}.
\end{align*}
Further, we give the definitions of Jacobi elliptic functions $\nc (u)$, $\dc (u), \sd(u)$ and $\ds (u)$, as follows:
\begin{align*}
\nc (u):=\frac{1}{\cn (u)}, \quad \dc (u):=\frac{\dn (u)}{\cn (u)}, \quad \ds (u):=\frac{\dn (u)}{\sn (u)},\quad\sd(u):=\frac{\sn (u)}{\dn (u)}.
\end{align*}

We adopt the following notations:
\begin{defn}\label{de1}(\cite[Defn. 1.1]{XuZhao-2022}) Let $m\in\N$ and $p\in\Z$. Define
\begin{align*}
&S_{p,m}(y):=\sum\limits_{n=1}^\infty \frac {n^p}{\sinh^m(ny)},\quad {\bar S}_{p,m}(y):=\sum\limits_{n=1}^\infty \frac {n^p}{\sinh^m(ny)}(-1)^{n-1},\\
&C_{p,m}(y):=\sum\limits_{n=1}^\infty \frac {n^p}{\cosh^m(ny)},\quad {\bar C}_{p,m}(y):=\sum\limits_{n=1}^\infty \frac {n^p}{\cosh^m(ny)}(-1)^{n-1},\\
&S'_{p,m}(y):=\sum\limits_{n=1}^\infty \frac {\tn^p}{\sinh^m(\tn y)},\quad \widetilde{S}_{p,m}(y):=\sum\limits_{n=1}^\infty \frac {\tn^p}{\sinh^m(\tn y)}(-1)^{n-1},\\
&C'_{p,m}(y):=\sum\limits_{n=1}^\infty \frac {\tn^p}{\cosh^m(\tn y)},\quad \widetilde{C}_{p,m}(y):=\sum\limits_{n=1}^\infty \frac {\tn^p}{\cosh^m(\tn y)}(-1)^{n-1}.
\end{align*}
\end{defn}

Some recent results on infinite series involving hyperbolic functions may also be found in the works of \cite{AB2009,AB2013,B1977,B1978,BB2002,Campbell,KMT2014,T2015,T2008,T2010,T2012,X2018,XuZhao-2022,XZ2023,Ya-2018}. For example, the closed form of $C'_{2,2}(\pi)$ can be found in Berndt-Bialek-Yee's paper \cite[Corollary 3.9]{BB2002} and Andrews-Berndt's book \cite[Page 233]{AB2009}.

We need the following lemmas.
\begin{lem} For any integer $k\geq 0$ and $x\in \mathbb{C}$,
\begin{align}
&\frac{d^{2k}}{dx^{2k}} \left(\frac{1}{\sinh(x)}\right)=\sum_{l=0}^k \frac{B_{k,l}}{\sinh^{2l+1}(x)},\label{equ-sinh-single-one}\\
&\frac{d^{2k}}{dx^{2k}} \left(\frac{1}{\cosh(x)}\right)=\sum_{l=0}^k (-1)^l \frac{B_{k,l}}{\cosh^{2l+1}(x)},\label{equ-cosh-single-one}
\end{align}
where the coefficients $B_{k,l} \ (B_{k,0}:=1,B_{k,k}=(2k)!)$ have the following recurrence relation:
\begin{align*}
B_{k,l}=(2l-1)(2l)B_{k-1,l-1}+(2l+1)^2B_{k-1,l}
\end{align*}
for $1\leq l\leq k$.
\end{lem}
\begin{proof}
The Lemma follows immediately from the \cite[Eqs. (2.7)-(2.8)]{X2020CKMS}.
\end{proof}

\begin{lem}\emph{(cf. \cite{L1974})}
For any integer $k\geq 0$ and $x\in \mathbb{C}$,
\begin{align}
&\frac1{(2k+1)!}\frac{d^{2k}}{dx^{2k}}\left(\frac{1}{\sinh^2(x)}\right)=\sum\limits_{l=0}^{k} \frac{A_{2k+2,2l+2}}{\sinh^{2l+2}(x)},\label{equ-diff-sinh-fun-two}\\
&\frac1{(2k+1)!}\frac{d^{2k}}{dx^{2k}}\left(\frac{1}{\cosh^2(x)}\right)=\sum\limits_{l=0}^{k} (-1)^{l} \frac{A_{2k+2,2l+2}}{\cosh^{2l+2}(x)},\label{equ-diff-cosh-fun-two}
\end{align}
where the coefficients $A_{2k,2l} \ (A_{2k,2k}:=1)$ have the following recurrence relation:
\begin{align*}
A_{2k+2,2l}=\frac{1}{2k(2k+2)}\left((2l-1)(2l-2)A_{2k,2l-2}+4l^2A_{2k,2l}\right)
\end{align*}
for $1\leq l\leq k$. See Table 5 of \cite{L1974} for some values of $A_{2k,2l}$.
\end{lem}

\begin{lem}\label{lem-sigma-polyn-criterion}
Let $\gs=x(1-x)$. If $f(x)\in\Q[x]$ then
\begin{align*}	
f(x)\in \Q[\gs] \Longleftrightarrow &\,f(x)-f(1-x)=0, \\
f(x)\in \Q[\gs]\gs' \Longleftrightarrow &\, f(x)+f(1-x)=0.
\end{align*}
\end{lem}
\begin{proof}
If $f(x)\in \Q[\gs]$ then clearly $f(x)=f(1-x)$.
Similarly, if $g(x)\in \Q[\gs]\gs'$ then $f(x)=-f(1-x)$ since $\gs'(1-x)=2x-1=-\gs'(x)$.

To prove the converse statements, first suppose $f(x)=f(1-x)$. Let the highest degree term of $f(x)$ be $c x^D$.
Then the highest degree term of $f(1-x)$ is $c(-1)^D x^D$. Thus $D$ must be even, say $D=2d$.
Then we see that
$\tilde{f}(x):=f(x)-c(-\gs)^{d}$ has smaller degree and still satisfies
$\tilde{f}(x)=\tilde{f}(1-x)$. By induction we see easily that $f(x)\in\Q[\gs]$.

Now we assume $f(x)+f(1-x)=0$. Let $h(x)=f(x)+f(1-x)$.
If the leading term of $f(x)$ is $cx^D$ then the leading term of $f(1-x)$ is $c(-1)^D x^D$.
Thus $D=2d+1$ must be odd. It follows that
$\tilde{f}(x):=f(x)+c(-\gs)^{d}\gs'/2$ has smaller degree and still satisfies
$\tilde{f}(x)+\tilde{f}(1-x)=0$ since $\gs'(1-x)=2x-1=-\gs'(x)$.
By induction we must have $f(x)\in\Q[\gs]\gs'$.

We have now completed the proof of the lemma.
\end{proof}

\begin{lem}\label{lem-one-exform-sinh-Qxz}
For any integer $\ss\geq 1$,
\begin{align*}	
{\bar S}_{4\ss+1,1}(y)\in z^{4\ss+2}\Q[\gs]\gs',\quad \text{and} \quad
{\bar S}_{4\ss-1,1}(y)\in z^{4\ss}\Q[\gs].
\end{align*}
\end{lem}

\begin{proof}
By \cite[Theorem 3.14]{RXZ2023} we see that for any integer $m>0$,
\begin{equation}\label{equ:accurateBarS}
{\bar S}_{2\ss+1,1}(y)=\sum_{n=1}^{\infty}\frac{(-1)^{n-1}n^{2\ss+1}}{\sinh(ny)}=\frac{(2\ss)!}{2^{2\ss+2}}z^{2\ss+2}\gs R_{2\ss}(x),
\end{equation}
where
\begin{equation*}
R_{2\ss}(x)=\frac{(x-1)^{s-1}}{(2\ss)!} r_{2\ss}\Big(\frac{x}{x-1}\Big)\in \Q[x],
\end{equation*}
where  $r_{2n}(x)$ appears in the coefficients of the Maclaurin series of
$\sn^2(u,k)$ (see \cite[Lemma 3.11]{RXZ2023} where it was denoted by $q_{2n}(x)$) satisfying
\begin{equation*}
r_{2\ss}(x)=x^{s-1} r_{2\ss}(1/x).
\end{equation*}
Therefore
\begin{align*}
R_{2\ss}(1-x)=&\, \frac{(-1)^{\ss-1} x^{\ss-1}}{(2\ss)!} r_{2\ss}\Big(\frac{x-1}{x}\Big) \\
=&\, \frac{(-1)^{s-1} x^{\ss-1}}{(2\ss)!}\Big(\frac{x-1}{x}\Big)^{\ss-1}  r_{2\ss}\Big(\frac{x}{x-1}\Big)=(-1)^{\ss-1} R_{2\ss}(x).
\end{align*}
Hence for all $\ss\geq 1$
\begin{equation*}
R_{4\ss}(x)+R_{4\ss}(1-x)=0, \quad\text{and}   \quad R_{4\ss-2}(x)-R_{4\ss-2}(1-x)=0.
\end{equation*}
Therefore, by Lemma~\ref{lem-sigma-polyn-criterion},
\begin{align*}		
R_{4\ss}(x)\in \Q[\gs]\gs', \quad \text{and} \quad
R_{4\ss-2} \in \Q[\gs].
\end{align*}
The lemma follows quickly from \eqref{equ:accurateBarS}.
\end{proof}

\begin{lem}\label{lem-one-exform-cosh-Qx/2z}
Let $v=\sqrt{x(1-x)}$. For any integer $\ss\ge0$,
\begin{align*}		
\widetilde{C}_{4\ss+1,1}(y)\in  z^{4\ss+2}  \Q[v^2]v, \quad \text{and} \quad
\widetilde{C}_{4\ss+3,1}(y)\in  z^{4\ss+4}  \Q[v^2]v^2 v'.
\end{align*}
\end{lem}

\begin{proof}
From  \cite[Lemma 3.1 and Theorem 3.2]{RXZ2023} we see that
\begin{equation}\label{equ:widetildeCbyp(x)}
\widetilde{C}_{2\ss+1,1}(y)=\sum_{n=1}^{\infty}\frac{(-1)^{n-1}\tn^{2\ss+1}}{\cosh\big( \tn y\big)}
=\frac{(-1)^\ss}{2^{2\ss+2}}z^{2\ss+2}  p_{2\ss+1}(x) v ,
\end{equation}
where the
\begin{equation*}
\sum_{n\geq 0} \frac{p_n(x)}{n!} u^n
\end{equation*}
is the Maclaurin series of $\sd(u)$.  We have the following classical result
(see the bottom of page 47 of Hancock's book \cite{Hancock1910}):
\begin{equation}\label{equ:Hancock}
\sn(iu,\sqrt{1-k^2})=i\frac{\sn(u,k)}{\cn(u,k)},\quad   \dn(iu,\sqrt{1-k^2})=\frac{\dn(u,k)}{\cn(u,k)},
\end{equation}
where $i=\sqrt{-1}$. Therefore
\begin{align*}
i\sd(iu,\sqrt{1-k^2})+\sd(u,k)=0.	
\end{align*}
This easily implies that $p_{2\ss}(x)=0$,
\begin{align*}
p_{4\ss+1}(1-x)-p_{4\ss+1}(x)=0,\quad \text{and} \quad p_{4\ss+3}(1-x)+p_{4\ss+3}(x)=0
\end{align*}
for all $m\geq 0$. By \eqref{equ:widetildeCbyp(x)} and Lemma~\ref{lem-sigma-polyn-criterion}
\begin{align*}		
\widetilde{C}_{4\ss+1,1}(y)\in  z^{4\ss+2} \Q[v^2]v, \quad \text{and} \quad
\widetilde{C}_{4\ss+3,1}(y)\in  z^{4\ss+4} \Q[v^2]v\gs'.
\end{align*}
The lemma follows quickly from the fact that $v v'=\gs'/2$.
\end{proof}

\begin{lem}\label{lem-one-exform-cosh-Qxz}
For any integer $\ss\geq 1$ we have
\begin{align*}		
C'_{4\ss,2}(y)\in z^{4\ss+2}\Q[\gs]\gs' +z^{4\ss+1}z' \Q[\gs], \quad \text{and} \quad
C'_{4\ss-2,2}(y)\in  z^{4\ss}\Q[\gs] +z^{4\ss-1}z' \Q[\gs]\gs'.
\end{align*}
\end{lem}

\begin{proof}
By \cite[Cor.~7.4 and Cor.~8.9]{XZ2023}, for any integer $m\geq 1$ we have
\begin{equation}\label{equ:C'byq(x)}
C'_{2\ss,2}(y)=\sum_{n=1}^\infty \frac{\tn^{2\ss}}{\cosh^2(\tn y)}
 =\frac{(-1)^{\ss}}{2^{2\ss+1} m} \gs z^{2\ss+1}\Big(z q'_{2\ss}(x)+2\ss z'q_{2\ss}(x)\Big),
\end{equation}
where
\begin{equation*}
\sum_{n\geq 0} \frac{q_n(x)}{n!} u^n
\end{equation*}
is the Maclaurin series of $u\ds(u)$.  By \eqref{equ:Hancock} we see that
\begin{align*}
iu\ds(iu,\sqrt{1-k^2})-u\ds(u,k)=0.	
\end{align*}
This easily implies that $q_{2\ss+1}(x)=0$,
\begin{align*}
q_{4\ss}(1-x)-q_{4\ss}(x)=0,\quad \text{and} \quad q_{4\ss+2}(1-x)+q_{4\ss+2}(x)=0
\end{align*}
for all $m\geq 0$. The lemma follows immediately from \eqref{equ:C'byq(x)} and Lemma~\ref{lem-sigma-polyn-criterion}.
\end{proof}

\begin{lem}\label{lem-one-noexform-sinh-Qxz}
\emph{(cf. \cite{XZ2023})}
Let $\gs=x(1-x)$. For any integer $\ss\geq 1$ we have
\begin{align*}
S_{2,2}(y)  &\, \in z^4\Q[\gs]+z^3z'\Q[\gs]\gs'+z^2(z')^2\Q[\gs] ,\\
S_{4\ss+2,2}(y) &\, \in z^{4\ss+4}\Q[\gs]+z^{4\ss+3}z'\Q[\gs]\gs',   \\
S_{4\ss,2}(y) &\, \in z^{4\ss+2}\Q[\gs]\gs'+z^{4\ss+1}z'\Q[\gs].
\end{align*}
\end{lem}
\begin{proof}
This follows immediately from \cite[Prop. 3.2]{XZ2023} generalizing a result of Ramanujan \cite[Page 142, Eq. (29)]{Rama1916}. Note that
\begin{equation*}
S_{2\ss,2}(y)=\sum_{n=1}^\infty \frac{n^{2\ss}}{\sinh^2(ny)} =\sum_{n=1}^\infty \frac{4q^{2n} n^{2\ss}}{(1-q^{2n})^2} =4\Phi_{1,2\ss}=4\Phi_{1,2\ss}(q^2)
\end{equation*}
in the notation of \cite{Rama1916,XZ2023}.
\end{proof}

For convenience, let $\widetilde{B}_{k,l}:=(-1)^l B_{k,l}$, $D_{k,l}:=(2k+1)!A_{2k+2,2l+2}$ and $\widetilde{D}_{k,l}:=(-1)^l(2k+1)!A_{2k+2,2l+2}=(-1)^lD_{k,l}\ (0\leq l\leq k)$.
We define four $(k+1)\times (k+1)$ matrixs ${\bf B}_{k+1}, {\bf \widetilde{B}}_{k+1},{\bf D}_{k+1}$ and ${\bf \widetilde{D}}_{k+1}$ by
\begin{align}
&{\bf B}_{k+1}:=\left( {\begin{array}{*{20}{c}}
{{B_{0,0}}}&0&0& \cdots &0\\
{{B_{1,0}}}&{{B_{1,1}}}&0& \cdots &0\\
{{B_{2,0}}}&{{B_{2,1}}}&{{B_{2,2}}}& \cdots &0\\
 \vdots & \vdots & \vdots & \ddots & \vdots \\
{{B_{k,0}}}&{{B_{k,1}}}&{{B_{k,2}}}& \cdots &{{B_{k,k}}}
\end{array}} \right),\quad {\bf \widetilde{B}}_{k+1}:=\left( {\begin{array}{*{20}{c}}
{{\widetilde{B}_{0,0}}}&0&0& \cdots &0\\
{{\widetilde{B}_{1,0}}}&{{\widetilde{B}_{1,1}}}&0& \cdots &0\\
{{\widetilde{B}_{2,0}}}&{{\widetilde{B}_{2,1}}}&{{\widetilde{B}_{2,2}}}& \cdots &0\\
 \vdots & \vdots & \vdots & \ddots & \vdots \\
{{\widetilde{B}_{k,0}}}&{{\widetilde{B}_{k,1}}}&{{\widetilde{B}_{k,2}}}& \cdots &{{\widetilde{B}_{k,k}}}
\end{array}} \right),\label{defn-matrx-orig-one}\\
&{\bf D}_{k+1}:=\left( {\begin{array}{*{20}{c}}
{{D_{0,0}}}&0&0& \cdots &0\\
{{D_{1,0}}}&{{D_{1,1}}}&0& \cdots &0\\
{{D_{2,0}}}&{{D_{2,1}}}&{{D_{2,2}}}& \cdots &0\\
 \vdots & \vdots & \vdots & \ddots & \vdots \\
{{D_{k,0}}}&{{D_{k,1}}}&{{D_{k,2}}}& \cdots &{{D_{k,k}}}
\end{array}} \right),\quad {\bf \widetilde{D}}_{k+1}:=\left( {\begin{array}{*{20}{c}}
{{\widetilde{D}_{0,0}}}&0&0& \cdots &0\\
{{\widetilde{D}_{1,0}}}&{{\widetilde{D}_{1,1}}}&0& \cdots &0\\
{{\widetilde{D}_{2,0}}}&{{\widetilde{D}_{2,1}}}&{{\widetilde{D}_{2,2}}}& \cdots &0\\
 \vdots & \vdots & \vdots & \ddots & \vdots \\
{{\widetilde{D}_{k,0}}}&{{\widetilde{D}_{k,1}}}&{{\widetilde{D}_{k,2}}}& \cdots &{{\widetilde{D}_{k,k}}}
\end{array}} \right).\label{defn-matrx-orig-two}
\end{align}
From linear algebra we know that the matrix $\left( {\begin{array}{*{20}{c}}
A&0\\
C&B
\end{array}} \right)$ is invertible with
\[{\left( {\begin{array}{*{20}{c}}
A&0\\
C&B
\end{array}} \right)^{ - 1}} = \left( {\begin{array}{*{20}{c}}
{{A^{ - 1}}}&0\\
{ - {B^{ - 1}}C{A^{ - 1}}}&{{B^{ - 1}}}
\end{array}} \right).\]
Therefore, the four inverse matrixs ${\bf B}_{k+1}^{ - 1},{\bf \widetilde{B}}_{k+1}^{ - 1},{\bf D}_{k+1}^{ - 1}$ and ${\bf \widetilde{D}}_{k+1}^{ - 1}$ can be written in the following forms
\begin{align}
&{\bf B}_{k+1}^{ - 1}=\left( {\begin{array}{*{20}{c}}
{{b_{0,0}}}&0&0& \cdots &0\\
{{b_{1,0}}}&{{b_{1,1}}}&0& \cdots &0\\
{{b_{2,0}}}&{{b_{2,1}}}&{{b_{2,2}}}& \cdots &0\\
 \vdots & \vdots & \vdots & \ddots & \vdots \\
{{b_{k,0}}}&{{b_{k,1}}}&{{b_{k,2}}}& \cdots &{{b_{k,k}}}
\end{array}} \right),\quad {\bf \widetilde{B}}_{k+1}^{ - 1}=\left( {\begin{array}{*{20}{c}}
{{\widetilde{b}_{0,0}}}&0&0& \cdots &0\\
{{\widetilde{b}_{1,0}}}&{{\widetilde{b}_{1,1}}}&0& \cdots &0\\
{{\widetilde{b}_{2,0}}}&{{\widetilde{b}_{2,1}}}&{{\widetilde{b}_{2,2}}}& \cdots &0\\
 \vdots & \vdots & \vdots & \ddots & \vdots \\
{{\widetilde{b}_{k,0}}}&{{\widetilde{b}_{k,1}}}&{{\widetilde{b}_{k,2}}}& \cdots &{{\widetilde{b}_{k,k}}}
\end{array}} \right),\label{defn-matrx-one}\\
&{\bf D}_{k+1}^{ - 1}=\left( {\begin{array}{*{20}{c}}
{{d_{0,0}}}&0&0& \cdots &0\\
{{d_{1,0}}}&{{d_{1,1}}}&0& \cdots &0\\
{{d_{2,0}}}&{{d_{2,1}}}&{{d_{2,2}}}& \cdots &0\\
 \vdots & \vdots & \vdots & \ddots & \vdots \\
{{d_{k,0}}}&{{d_{k,1}}}&{{d_{k,2}}}& \cdots &{{d_{k,k}}}
\end{array}} \right),\quad {\bf \widetilde{D}}_{k+1}^{ - 1}=\left( {\begin{array}{*{20}{c}}
{{\widetilde{d}_{0,0}}}&0&0& \cdots &0\\
{{\widetilde{d}_{1,0}}}&{{\widetilde{d}_{1,1}}}&0& \cdots &0\\
{{\widetilde{d}_{2,0}}}&{{\widetilde{d}_{2,1}}}&{{\widetilde{d}_{2,2}}}& \cdots &0\\
 \vdots & \vdots & \vdots & \ddots & \vdots \\
{{\widetilde{d}_{k,0}}}&{{\widetilde{d}_{k,1}}}&{{\widetilde{d}_{k,2}}}& \cdots &{{\widetilde{d}_{k,k}}}
\end{array}} \right),\label{defn-matrx-two}
\end{align}
where $b_{j,l},\widetilde{b}_{j,l},d_{j,l},\widetilde{d}_{j,l}\in \Q$ for $0\leq j,l\leq k$.

\begin{thm}\label{thm-alter-Sinh-one}
For any integers $m\geq 1$ and $k\geq 0$, we have
\begin{align*}
{\bar S}_{2(m+k)-1,2k+1}(y)=\sum_{j=0}^{k}b_j \frac{d^{2j}}{dy^{2j}}\Big({\bar S}_{2(m+k-j)-1}(y)\Big) ,
\end{align*}
where ${\bar S}_p(y):={\bar S}_{p,1}(y)$ and $b_j\in \Q$.
\end{thm}
\begin{proof}
Using \eqref{equ-sinh-single-one}, by a direct calculation, one obtains
\begin{align*}
\frac{d^{2k}}{dy^{2k}}  \left(\sum_{n=1}^\infty\frac{(-1)^{n-1}n^{2m-1}}{\sinh(ny)}\right)=\sum_{l=0}^k B_{k,l}\sum_{n=1}^\infty\frac{(-1)^{n-1}n^{2m+2k-1}}{\sinh^{2l+1}(ny)}.
\end{align*}
Thus,
\begin{align*}
\left({\bar S}_{2m-1}(y)\right)^{(2k)}:=\frac{d^{2k}}{dy^{2k}}{\bar S}_{2m-1}(y)=\sum_{l=0}^k B_{k,l}{\bar S}_{2m+2k-1,2l+1}(y).
\end{align*}
Hence, we have
\[\left( {\begin{array}{*{20}{c}}
{\bar S_{2(m + k) - 1}^{\left( 0 \right)}\left( y \right)}\\
{\bar S_{2(m + k) - 3}^{\left( 2 \right)}\left( y \right)}\\
{\bar S_{2(m + k) - 5}^{\left( 4 \right)}\left( y \right)}\\
 \vdots \\
{\bar S_{2m - 1}^{\left( {2k} \right)}\left( y \right)}
\end{array}} \right) = {\bf B}_{k+1}\left( {\begin{array}{*{20}{c}}
{{{\bar S}_{2(m + k) - 1,1}}\left( y \right)}\\
{{{\bar S}_{2(m + k) - 1,3}}\left( y \right)}\\
{{{\bar S}_{2(m + k) - 1,5}}\left( y \right)}\\
 \vdots \\
{{{\bar S}_{2(m + k) - 1,2k + 1}}\left( y \right)}
\end{array}} \right).\]
Then,
\begin{align}\label{matrix-one-Sinh-funct}
\left( {\begin{array}{*{20}{c}}
{{{\bar S}_{2(m + k) - 1,1}}\left( y \right)}\\
{{{\bar S}_{2(m + k) - 1,3}}\left( y \right)}\\
{{{\bar S}_{2(m + k) - 1,5}}\left( y \right)}\\
 \vdots \\
{{{\bar S}_{2(m + k) - 1,2k + 1}}\left( y \right)}
\end{array}} \right) = {\bf B}_{k+1}^{ - 1}\left( {\begin{array}{*{20}{c}}
{\bar S_{2(m + k) - 1}^{\left( 0 \right)}\left( y \right)}\\
{\bar S_{2(m + k) - 3}^{\left( 2 \right)}\left( y \right)}\\
{\bar S_{2(m + k) - 5}^{\left( 4 \right)}\left( y \right)}\\
 \vdots \\
{\bar S_{2m - 1}^{\left( {2k} \right)}\left( y \right)}
\end{array}} \right).
\end{align}
Finally, setting $b_j:=b_{k,j}$, and applying \eqref{defn-matrx-one} and \eqref{matrix-one-Sinh-funct} gives the desired result.
\end{proof}

\begin{thm}\label{thm-alter-Cosh-one}
For any integers $m\geq 1$ and $k\geq 0$, we have
\begin{align}
\widetilde{C}_{2(m+k)-1,2k+1}(y)=\sum_{j=0}^{k}\widetilde{b}_j\frac{d^{2j}}{dy^{2j}}\Big(\widetilde{C}_{2(m+k-j)-1}(y)\Big) ,
\end{align}
where $\widetilde{C}_p(y):=\widetilde{C}_{p,1}(y)$ and $\widetilde{b}_j\in \Q$.
\end{thm}
\begin{proof}
From \eqref{equ-cosh-single-one}, one obtains
\begin{align*}
\frac{d^{2k}}{dy^{2k}}  \left(\sum_{n=1}^\infty\frac{(-1)^{n-1}\tn^{2m-1}}{\cosh(\tn y)}\right)=\sum_{l=0}^k \widetilde{B}_{k,l}\sum_{n=1}^\infty\frac{(-1)^{n-1}\tn^{2m+2k-1}}{\cosh^{2l+1}(\tn y)}.
\end{align*}
Thus,
\begin{align*}
\left(\widetilde{C}_{2m-1}(y)\right)^{(2k)}:=\frac{d^{2k}}{dy^{2k}}\widetilde{C}_{2m-1}(y)=\sum_{l=0}^k \widetilde{B}_{k,l}\widetilde{C}_{2m+2k-1,2l+1}(y).
\end{align*}
By an elementary calculation, we have
\[\left( {\begin{array}{*{20}{c}}
{\widetilde{C}_{2(m + k) - 1}^{\left( 0 \right)}\left( y \right)}\\
{\widetilde{C}_{2(m + k) - 3}^{\left( 2 \right)}\left( y \right)}\\
{\widetilde{C}_{2(m + k) - 5}^{\left( 4 \right)}\left( y \right)}\\
 \vdots \\
{\widetilde{C}_{2m - 1}^{\left( {2k} \right)}\left( y \right)}
\end{array}} \right) = {\bf \widetilde{B}}_{k+1}\left( {\begin{array}{*{20}{c}}
{{\widetilde{C}_{2(m + k) - 1,1}}\left( y \right)}\\
{{\widetilde{C}_{2(m + k) - 1,3}}\left( y \right)}\\
{{\widetilde{C}_{2(m + k) - 1,5}}\left( y \right)}\\
 \vdots \\
{{\widetilde{C}_{2(m + k) - 1,2k + 1}}\left( y \right)}
\end{array}} \right).\]
Hence, we obtain
\begin{align}\label{matrix-one-Cosh-funct}
\left( {\begin{array}{*{20}{c}}
{{\widetilde{C}_{2(m + k) - 1,1}}\left( y \right)}\\
{{\widetilde{C}_{2(m + k) - 1,3}}\left( y \right)}\\
{{\widetilde{C}_{2(m + k) - 1,5}}\left( y \right)}\\
 \vdots \\
{{\widetilde{C}_{2(m + k) - 1,2k + 1}}\left( y \right)}
\end{array}} \right)={\bf \widetilde{B}^{-1}}_{k+1}\left( {\begin{array}{*{20}{c}}
{\widetilde{C}_{2(m + k) - 1}^{\left( 0 \right)}\left( y \right)}\\
{\widetilde{C}_{2(m + k) - 3}^{\left( 2 \right)}\left( y \right)}\\
{\widetilde{C}_{2(m + k) - 5}^{\left( 4 \right)}\left( y \right)}\\
 \vdots \\
{\widetilde{C}_{2m - 1}^{\left( {2k} \right)}\left( y \right)}
\end{array}} \right).
\end{align}
Finally, setting $\widetilde{b}_j:=\widetilde{b}_{k,j}$, and applying \eqref{defn-matrx-one} and \eqref{matrix-one-Cosh-funct} gives the desired result.
\end{proof}

\begin{thm}\label{thm-noalter-Sinh-one}
For any integers $m\geq 1$ and $k\geq 0$, we have
\begin{align}
S_{2(m+k),2k+2}(y)=\sum_{j=0}^{k}d_j\frac{d^{2j}}{dy^{2j}}\Big(S_{2(m+k-j),2}(y)\Big) ,
\end{align}
where $d_j\in \Q$.
\end{thm}
\begin{proof}
Using \eqref{equ-diff-sinh-fun-two}, by a direct calculation, one obtains
\begin{align*}
\frac{d^{2k}}{dy^{2k}}  \left(\sum_{n=1}^\infty\frac{n^{2m}}{\sinh^2(ny)}\right)=\sum_{l=0}^k D_{k,l}\sum_{n=1}^\infty\frac{n^{2m+2k}}{\sinh^{2l+2}(ny)}.
\end{align*}
Thus,
\begin{align*}
\left(S_{2m,2}(y)\right)^{(2k)}:=\frac{d^{2k}}{dy^{2k}}S_{2m,2}(y)=\sum_{l=0}^k D_{k,l}S_{2m+2k,2l+2}(y).
\end{align*}
Hence, we have
\[\left( {\begin{array}{*{20}{c}}
{S_{2(m + k),2}^{\left( 0 \right)}\left( y \right)}\\
{S_{2(m + k - 1),2}^{\left( 2 \right)}\left( y \right)}\\
{S_{2(m + k - 2),2}^{\left( 4 \right)}\left( y \right)}\\
 \vdots \\
{S_{2m,2}^{\left( {2k} \right)}\left( y \right)}
\end{array}} \right) = {\bf D}_{k+1}\left( {\begin{array}{*{20}{c}}
{{S_{2(m + k),2}}\left( y \right)}\\
{{S_{2(m + k),4}}\left( y \right)}\\
{{S_{2(m + k),6}}\left( y \right)}\\
 \vdots \\
{{S_{2(m + k),2k + 2}}\left( y \right)}
\end{array}} \right).\]
Then,
\begin{align}\label{matrix-two-Sinh-funct}
\left( {\begin{array}{*{20}{c}}
{{S_{2(m + k),2}}\left( y \right)}\\
{{S_{2(m + k),4}}\left( y \right)}\\
{{S_{2(m + k),6}}\left( y \right)}\\
 \vdots \\
{{S_{2(m + k),2k + 2}}\left( y \right)}
\end{array}} \right)={\bf D}^{-1}_{k+1}\left( {\begin{array}{*{20}{c}}
{S_{2(m + k),2}^{\left( 0 \right)}\left( y \right)}\\
{S_{2(m + k - 1),2}^{\left( 2 \right)}\left( y \right)}\\
{S_{2(m + k - 2),2}^{\left( 4 \right)}\left( y \right)}\\
 \vdots \\
{S_{2m,2}^{\left( {2k} \right)}\left( y \right)}
\end{array}} \right).
\end{align}
Finally, setting $d_j:=d_{k,j}$, and applying \eqref{defn-matrx-two} and \eqref{matrix-two-Sinh-funct} gives the desired result.
\end{proof}

\begin{thm}\label{thm-noalter-Cosh-one}
For any integers $m\geq 1$ and $k\geq 0$, we have
\begin{align}
C'_{2(m+k),2k+2}(y)=\sum_{j=0}^{k}\widetilde{d}_j\frac{d^{2j}}{dy^{2j}}\Big(C'_{2(m+k-j),2}(y)\Big),
\end{align}
where $\widetilde{d}_j\in \Q$.
\end{thm}
\begin{proof}
From \eqref{equ-diff-cosh-fun-two}, we have
\begin{align*}
\frac{d^{2k}}{dy^{2k}}  \left(\sum_{n=1}^\infty\frac{\tn^{2m}}{\cosh^2(\tn y)}\right)=\sum_{l=0}^k \widetilde{D}_{k,l}\sum_{n=1}^\infty\frac{\widetilde{n}^{2m+2k}}{\cosh^{2l+2}(\widetilde{n}y)}.
\end{align*}
Thus,
\begin{align*}
\left(C'_{2m,2}(y)\right)^{(2k)}:=\frac{d^{2k}}{dy^{2k}}C'_{2m,2}(y)=\sum_{l=0}^k \widetilde{D}_{k,l}C'_{2m+2k,2l+2}(y).
\end{align*}
Hence, we have
\[\left( {\begin{array}{*{20}{c}}
{{C'}_{2(m + k),2}^{\left( 0 \right)}\left( y \right)}\\
{{C'}_{2(m + k - 1),2}^{\left( 2 \right)}\left( y \right)}\\
{{C'}_{2(m + k - 2),2}^{\left( 4 \right)}\left( y \right)}\\
 \vdots \\
{{C'}_{2m,2}^{\left( {2k} \right)}\left( y \right)}
\end{array}} \right) = {\bf \widetilde{D}}_{k+1}\left( {\begin{array}{*{20}{c}}
{{C'_{2(m + k),2}}\left( y \right)}\\
{{C'_{2(m + k),4}}\left( y \right)}\\
{{C'_{2(m + k),6}}\left( y \right)}\\
 \vdots \\
{{C'_{2(m + k),2k + 2}}\left( y \right)}
\end{array}} \right).\]
Then,
\begin{align}\label{matrix-two-Cosh-funct}
\left( {\begin{array}{*{20}{c}}
{{C'_{2(m + k),2}}\left( y \right)}\\
{{C'_{2(m + k),4}}\left( y \right)}\\
{{C'_{2(m + k),6}}\left( y \right)}\\
 \vdots \\
{{C'_{2(m + k),2k + 2}}\left( y \right)}
\end{array}} \right)={\bf \widetilde{D}}^{-1}_{k+1}\left( {\begin{array}{*{20}{c}}
{{C'}_{2(m + k),2}^{\left( 0 \right)}\left( y \right)}\\
{{C'}_{2(m + k - 1),2}^{\left( 2 \right)}\left( y \right)}\\
{{C'}_{2(m + k - 2),2}^{\left( 4 \right)}\left( y \right)}\\
 \vdots \\
{{C'}_{2m,2}^{\left( {2k} \right)}\left( y \right)}
\end{array}} \right).
\end{align}
Finally, setting $\widetilde{d}_j:=\widetilde{d}_{k,j}$, and applying \eqref{defn-matrx-two} and \eqref{matrix-two-Cosh-funct} gives the desired result.
\end{proof}

\section{General Berndt-Type Integrals}
In this section, we will prove our main structural theorems on the general Berndt-type integrals \eqref{BTI-definition-1} with the denominator having arbitrary positive degrees.

\begin{lem}\label{lem:kthDerzn}
Let $k$ be a nonnegative integer. For any index $\bfi=(i_0,\dots,i_k)$ we write
\begin{equation*}
|\bfi|=\sum_{0\le j\le k}  i_j, \quad
\text{and} \quad
|\bfi_2|=\sum_{0\le j\le [k/2]} i_{2j}.
\end{equation*}
Put $v=\sqrt{x(1-x)}$, $v'=dv/dx$, and $z^{(j)}=d^jz/dx^j$.
Then for any rational function $g(v)\in\Q[v]$ we have
\begin{align}\label{equ:kthDerzn}
\frac{d^k}{dy^k} \Big(z^n g(v)\Big) \in  &\,
\left\{
  \begin{array}{ll}
      &\, z^{n+k} \sum_{\substack{i_0,\dots,i_k\geq 0\\ |\bfi|=k,\  |\bfi_2| \text{ even}}}
     \prod_{j=0}^k \big( z^{(j)} \big)^{i_j} \cdot  \Q\big[v,v^{-1}\big], \\
      &\, +z^{n+k} \sum_{\substack{i_0,\dots,i_k\geq 0\\ |\bfi|=k,\   |\bfi_2| \text{ odd}}}
     \prod_{j=0}^k \big( z^{(j)} \big)^{i_j}  \cdot \Q\big[v,v^{-1}\big]v';
  \end{array}
\right.
\\
\frac{d^k}{dy^k} \Big(z^n g(v)v'\Big) \in  &\,
\left\{
  \begin{array}{ll}
      &\, z^{n+k} \sum_{\substack{i_0,\dots,i_k\geq 0\\ |\bfi|=k,\  |\bfi_2| \text{ even}}}
     \prod_{j=0}^k \big( z^{(j)} \big)^{i_j} \cdot  \Q\big[v,v^{-1}\big]v', \\
      &\, +z^{n+k} \sum_{\substack{i_0,\dots,i_k\geq 0\\ |\bfi|=k,\   |\bfi_2| \text{ odd}}}
     \prod_{j=0}^k \big( z^{(j)} \big)^{i_j}  \cdot \Q\big[v,v^{-1}\big];\\
  \end{array}
\right. \label{equ:kthDerznDotu}
\\
     \frac{d^k}{dy^k} \Big(z^{n+1} z' g(v)\Big) \in  &\,
\left\{
  \begin{array}{ll}
      &\, z^{n+k} \sum_{\substack{i_0,\dots,i_{k+1}\geq 0\\ |\bfi|=k+1,\  |\bfi_2| \text{ odd}}}
     \prod_{j=0}^{k+1} \big( z^{(j)} \big)^{i_j} \cdot \Q\big[v,v^{-1}\big]v', \\
      &\, +z^{n+k} \sum_{\substack{i_0,\dots,i_{k+1}\geq 0\\ |\bfi|=k+1,\   |\bfi_2| \text{ even}}}
     \prod_{j=0}^{k+1} \big( z^{(j)} \big)^{i_j} \cdot \Q\big[v,v^{-1}\big];
  \end{array}
\right.  \label{equ:kthDerznz'}
\\
     \frac{d^k}{dy^k} \Big(z^{n+1} z' g(v)v'\Big) \in  &\,
\left\{
  \begin{array}{ll}
      &\, z^{n+k} \sum_{\substack{i_0,\dots,i_{k+1}\geq 0\\ |\bfi|=k+1,\  |\bfi_2| \text{ odd}}}
     \prod_{j=0}^{k+1} \big( z^{(j)} \big)^{i_j} \cdot \Q\big[v,v^{-1}\big], \\
      &\, +z^{n+k} \sum_{\substack{i_0,\dots,i_{k+1}\geq 0\\ |\bfi|=k+1,\   |\bfi_2| \text{ even}}}
     \prod_{j=0}^{k+1} \big( z^{(j)} \big)^{i_j} \cdot \Q\big[v,v^{-1}\big]v'.
  \end{array}
\right.  \label{equ:kthDerznz'Dotu}
\end{align}
\end{lem}

\begin{proof}
For \eqref{equ:kthDerzn}, we proceed by induction on $k$. If $k=0$ then the claim is trivial.
Suppose \eqref{equ:kthDerzn} holds for $k\geq 0$. Let's denote the two sums on the right-hand side
of \eqref{equ:kthDerzn} by $\text{I}_k$ and $\text{II}_k$.
Observe that (see \cite[P. 120, Entry. 9(i)]{B1991})
\begin{equation}\label{equ:dx/dy}
\frac{dx}{dy}=-x(1-x)z^2=-v^2 z^2.
\end{equation}
When $d/dy$ is applied to the one of the two factors separated by the center dots on the right-hand side of \eqref{equ:kthDerzn}
we have the following two cases to consider by the product rule:
\begin{enumerate}
  \item [\upshape{(i)}] $d/dy$ is applied to the $z$-block. First,
\begin{align*}
\frac{d}{dy} \big( z^{n+k+i_0} \big) =-(n+k+i_0)v z^{n+k+1+i_0}.
\end{align*}
Thus $|\bfi|$ and $|\bfi_2|$ are invariant in the expression of $\frac{d^{k+1}}{dy^{k+1}} \Big(z^nf(x)g(v)\Big)$.

Suppose $j\geq 1$. Since
\begin{align*}
\frac{d}{dx} \big( z^{(j)} \big)^{i_j}=i_j \big( z^{(j)} \big)^{i_j-1}  z^{(j+1)}
\end{align*}
we see that if $j$ is odd then both $i_0$ and $i_{j+1}$ is increased by 1 in the expression
of $\frac{d^{k+1}}{dy^{k+1}} \Big(z^n g(v)\Big)$. If $j$ is even then $i_0$ and $i_{j+1}$ is increased by 1
while $i_{j}$ is decreased by 1. In both cases the parity of $|\bfi_2|$ is invariant while $|\bfi|$ is increased by 1.

To summarize, in case (i) we always have $\text{I}_k \to \text{I}_{k+1}$ and $\text{II}_k \to \text{II}_{k+1}$.

  \item [\upshape{(ii)}] $d/dy$ is applied to the $(v,u)$-block. Suppose $h(v)\in \Q\big[v,v^{-1}\big]$. Since
\begin{align*}
&(v')^2=\left(\frac{1-2x}{2v}\right)^2=\frac{1-4v}{4v^2},\\
&v''(x)=\frac{d}{dx}\left(\frac{1-2x}{2v}\right)
=-\frac{2v+(1-2x)v'}{2v^2}=-\frac{4v^2-4v+1}{4v^3},
\end{align*}
we get
\begin{align*}
 \frac{d}{dy}\big( h(v) v')=&\, -v^2 z^2 \Big(\frac{dh}{dv} (v')^2+h(v) v'' \Big) \in  z^2 \Q\big[v,v^{-1}\big]\\
 \frac{d}{dy}\big( h(v))=&\, -v^2 z^2 \frac{dh}{dv} v'\in z^2 \Q\big[v,v^{-1}\big]v'.
\end{align*}
Therefore, in case (ii) we have $\text{I}_k \to \text{II}_{k+1}$ and $\text{II}_k \to \text{I}_{k+1}$.
\end{enumerate}

Combining (i) and (ii) we see that \eqref{equ:kthDerzn} holds by induction.
The proof of \eqref{equ:kthDerznDotu}--\eqref{equ:kthDerznz'Dotu} is completely similar and thus is left to the interested reader.
\end{proof}

\begin{cor}\label{cor:kthDerzn}
Let $\ss\in\N$. Then for any polynomial $g(v)\in \Q\big[v,v^{-1}\big]$
\begin{align}\label{equ:kthDerzn-at1/2}
    \frac{d^k}{dy^k} \Big(z^{2\ss} g(v) \Big)\Big|_{x=1/2}\in &\, \sum_{\gl=0}^{[k/2]}\sum_{j=k}^{3k}  \frac{\Gamma^{4\ss+8\gl}(1/4) }{\pi^{3\ss+j}}\Q,\\
 \frac{d^k}{dy^k} \Big( z^{2\ss} g(v)v'\Big)\Big|_{x=1/2}\in &\, \sum_{\gl=0}^{[(k-1)/2]}\sum_{j=k}^{3k}  \frac{\Gamma^{4\ss+8\gl+4}(1/4) }{\pi^{3\ss+j}}\Q, \label{equ:kthDerznDotu-at1/2}\\
     \frac{d^k}{dy^k} \Big(z^{2\ss+1}z' g(v)\Big)\Big|_{x=1/2}\in &\, \sum_{\gl=0}^{[k/2]}\sum_{j=k+1}^{3k+3}  \frac{\Gamma^{4\ss+8\gl}(1/4) }{\pi^{3\ss+j}}\Q, \label{equ:kthDerznz'-at1/2} \\
     \frac{d^k}{dy^k} \Big(z^{2\ss+1}z' g(v)v'\Big)\Big|_{x=1/2}\in &\, \sum_{\gl=0}^{[(k-1)/2]}\sum_{j=k+1}^{3k+3}  \frac{\Gamma^{4\ss+8\gl+4}(1/4) }{\pi^{3\ss+j}}\Q. \label{equ:kthDerznz'Dotu-at1/2}
\end{align}

\end{cor}
\begin{proof}
Note that for all $j\geq 0$
\begin{equation*}
z\Big(\frac12\Big)z^{(2j)}\Big(\frac12\Big)\in \frac{\Gamma^4(1/4)}{\pi^3} \Q,\quad
z\Big(\frac12\Big)z^{(2j+1)}\Big(\frac12\Big)\in \frac{1}{\pi} \Q
\end{equation*}
by \eqref{equ:znDer}. By \eqref{equ:kthDerzn} and noting that $v'(1/2)=0$ we see that
\begin{align*}
\frac{d^k}{dy^k} \Big(z^{2\ss}g(v) \Big)\Big|_{x=1/2} \in &\, z^{2\ss}\Big(\frac12\Big)
    \sum_{\substack{i_0,\dots,i_k\geq 0\\ |\bfi|=k,\  |\bfi_2| \text{ even}}}
     \prod_{j=0}^k \Big[ z\Big(\frac12\Big) z^{(j)}\Big(\frac12\Big) \Big]^{i_j} \Q\\
= &\,  \sum_{\substack{i_0,\dots,i_k\geq 0\\ |\bfi|=k,\  |\bfi_2| \text{ even}}}
 \frac{\Gamma^{4\ss+4(i_0+i_2+i_4+\cdots)}(1/4)}{\pi^{3\ss+3i_0+i_1+3i_2+i_3+\cdots}}\Q.
\end{align*}
Since
\begin{align*}
 0\le  i_0+i_2+i_4+\cdots &\, \le i_0+\cdots+i_k= k \le  3i_0+i_1+3i_2+i_3+\cdots \le  3(i_0+\cdots+i_k)=3k
\end{align*}
we obtain \eqref{equ:kthDerzn-at1/2} immediately. The proofs of
\eqref{equ:kthDerznDotu-at1/2}--\eqref{equ:kthDerznz'Dotu-at1/2} are completely similar and are thus
left to the interested reader.
\end{proof}

\begin{thm} \label{thm:Plus}
For all integers $m\geq 1$ and $p\geq [m/2]$, the Berndt-type integrals
\begin{align}
\int_0^\infty \frac{x^{4p+1} dx}{(\cos x+\cosh x)^{m}} \in   \Q\Big[\Gamma^4(1/4),\pi^{-1}\Big].
\end{align}
Moreover, the degrees of $\Gamma^4(1/4)$ have the same parity as $m$ and are between $2p-m+2$ and $2p+m$, inclusive, while
the degrees of $\pi^{-1}$ are between $2p-m+2$ and $2p+3m-2$, inclusive.
\end{thm}
\begin{proof}
From Theorem \ref{main-thm-one-plus}
\begin{align*}
\int_0^\infty \frac{x^{4p+1} dx}{(\cos x+\cosh x)^m}
\in \sum_{l=0}^{[(m-1)/2]}\sum_{j=0}^{m-1-2l}  \Q \pi^{4p+2-j}\frac{d^{m-1-2l-j}}{dy^{m-1-2l-j}}\left\{\sum_{n=1}^\infty (-1)^{mn} \frac{\tn^{4p+2+2l-m}}{\cosh^m (\tn y)}\right\}\bigg|_{x=1/2}.
\end{align*}
If $m$ is odd then for all integers $p\geq (m-1)/2$ Theorem \ref{thm-alter-Cosh-one} yields
\begin{align*}
\int_0^\infty \frac{x^{4p+1} dx}{(\cos x+\cosh x)^m}
\in &\,\sum_{l=0}^{(m-1)/2}\sum_{j=0}^{m-1-2l}  \Q \pi^{4p+2-j}\frac{d^{m-1-2l-j}}{dy^{m-1-2l-j}}
\widetilde{C}_{4p+2+2l-m,m}(y) \bigg|_{x=1/2} \\
=&\, \sum_{l=0}^{(m-1)/2}\sum_{j=0}^{m-1-2l}
\sum_{h=0}^{(m-1)/2}\Q \pi^{4p+2-j}\Big(\widetilde{C}_{4p+2+2l-m-2h}(y)\Big)^{(m-1-2l-j+2h)}  \bigg|_{x=1/2}.
\end{align*}
By Lemma \ref{lem-one-exform-cosh-Qx/2z}, setting $\mu=4p+3+2l-m-2h$ (which is even and $\ge 6$) we get
\begin{align*}		
\widetilde{C}_{\mu-1,1}(y)=\left\{
                             \begin{array}{ll}
                               \widetilde{C}_{4\ss+1,1}(y)\in  z^{4\ss+2}  \Q[v], \quad &  \quad\hbox{if $\mu=4\ss+2$;} \\
                               \widetilde{C}_{4\ss+3,1}(y)\in  z^{4\ss+4}  \Q[v] v', \quad &  \quad\hbox{if $\mu=4\ss+4$.}
                             \end{array}
                           \right.
\end{align*}
Set $\kappa=m-1-2l-j+2h$. If $\mu=4\ss+2$ then \eqref{equ:kthDerzn-at1/2}
\begin{equation*}
\frac{d^\kappa}{dy^\kappa}\Big(\widetilde{C}_{\mu-1}(y)\Big)\bigg|_{x=1/2} \in
\sum_{\gl=0}^{[\kappa/2]} \sum_{q=\kappa}^{3\kappa}  \frac{\Gamma^{8\ss+4+8\gl}(1/4) }{\pi^{3\mu/2+q}}\Q
=\sum_{\substack{0\le i\le \kappa\\  i\equiv 0 \mod{2}}}\sum_{q=\kappa}^{3\kappa} \frac{\Gamma^{2\mu+4i}(1/4) }{\pi^{3\mu/2+q-4p-2+j}}\Q
\end{equation*}
and if $\mu=4\ss+4$ then \eqref{equ:kthDerznDotu-at1/2} implies
\begin{equation*}
\frac{d^\kappa}{dy^\kappa}\Big(\widetilde{C}_{\mu-1}(y)\Big)\bigg|_{x=1/2} \in
\sum_{\gl=0}^{[(\kappa-1)/2]} \sum_{q=\kappa}^{3\kappa}  \frac{\Gamma^{8\ss+12+8\gl}(1/4) }{\pi^{3\mu/2+q}}\Q
=\sum_{\substack{0\le i\le \kappa\\  i\equiv 1 \mod{2}}}\sum_{q=\kappa}^{3\kappa} \frac{\Gamma^{2\mu+4i}(1/4) }{\pi^{3\mu/2+q-4p-2+j}}\Q.
\end{equation*}
Putting these two cases together, we have
\begin{align*}
\int_0^\infty \frac{x^{4p+1} dx}{(\cos x+\cosh x)^m}
\in &\, \sum_{l=0}^{(m-1)/2}\sum_{j=0}^{m-1-2l}
\sum_{h=0}^{(m-1)/2}\sum_{\substack{0\le i\le \kappa\\ 2\mu+4i\equiv 4 \mod{8}}} \sum_{q=\kappa}^{3\kappa}  \frac{\Gamma^{2\mu+4i}(1/4) }{\pi^{3\mu/2+q-4p-2+j}}\Q.
\end{align*}
Note that the power $2\mu+4i$ of $\Gamma$ also satisfies
\begin{equation*}
8p-4m+8\le 2\mu+4i \le 2\mu+4\kappa\le 8p+2m+2+4h\le 8p+4m,
\end{equation*}
while the power of $\pi^{-1}$ satisfies
\begin{align*}
\frac32\mu +q-4p-2+j=&\,\frac32(4p+3+2l-m-2h) +q-4p-2+j  \\
\geq &\, 2p+\frac52+3l-\frac32m-3h+\kappa+j\\
\geq &\, 2p+\frac52+3l-\frac32m-3h+(m-1-2l+2h) \\
\geq &\, 2p+\frac32-\frac12m-\frac12(m-1)
=2p-m+2
\end{align*}
and
\begin{align*}
\frac32\mu +q-4p-2+j
\le &\, 2p+\frac52+3l-\frac32m-3h+3(m-1-2l-j+2h)+j\\
\le &\, 2p-\frac12+\frac32m+\frac32(m-1)=2p+3m-2.
\end{align*}

Similarly, if $m$ is even then for all integers  $p\geq m/2$ Theorem \ref{thm-noalter-Cosh-one} yields
\begin{align*}
\int_0^\infty \frac{x^{4p+1} dx}{(\cos x+\cosh x)^m}
\in &\,\sum_{l=0}^{m/2-1}\sum_{j=0}^{m-1-2l}  \Q \pi^{4p+2-j}\frac{d^{m-1-2l-j}}{dy^{m-1-2l-j}}
C'_{4p+2+2l-m,m}(y) \bigg|_{x=1/2} \\
=&\, \sum_{l=0}^{m/2-1}\sum_{j=0}^{m-1-2l}
\sum_{h=0}^{m/2-1}\Q \pi^{4p+2-j}\Big(C'_{4p+2+2l-m-2h,2}(y)\Big)^{(m-1-2l-j+2h)}  \bigg|_{x=1/2}.
\end{align*}
By Lemma \ref{lem-one-exform-cosh-Qxz}, setting $\mu=4p+4+2l-m-2h$ (which is even and $\ge 2$) we get
\begin{align*}		
C'_{\mu-2,2}(y)=\left\{
                             \begin{array}{ll}
                              C'_{4\ss,2}(y)\in z^{4\ss+2}\Q[v]v' +z^{4\ss+1}z' \Q[v], \quad & \quad \hbox{if $\mu=4\ss+2$;} \\
                              C'_{4\ss-2,2}(y)\in z^{4\ss}\Q[v] +z^{4\ss-1}z' \Q[v] v', \quad & \quad \hbox{if $\mu=4\ss$.}
                             \end{array}
                           \right.
\end{align*}
Set $\kappa=m-1-2l-j+2h$. If $\mu=4\ss+2$ then \eqref{equ:kthDerznDotu-at1/2} and \eqref{equ:kthDerznz'-at1/2} yield
\begin{align*}
\frac{d^\kappa}{dy^\kappa}\Big(C'_{4\ss,2}(y)\Big)\bigg|_{x=1/2} \in &\,
\sum_{\gl=0}^{[(\kappa-1)/2]} \sum_{q=\kappa}^{3\kappa}  \frac{\Gamma^{8\ss+8+8\gl}(1/4) }{\pi^{6\ss+3+q}}\Q
+\sum_{\gl=0}^{[\kappa/2]} \sum_{q=\kappa+1}^{3\kappa+3}  \frac{\Gamma^{8\ss+8\gl}(1/4) }{\pi^{6\ss+q}}\Q  \\
\subset &\,  \sum_{\gl=-1}^{[\kappa/2]} \sum_{q=\kappa-2}^{3\kappa}  \frac{\Gamma^{8\ss+8+8\gl}(1/4) }{\pi^{3\mu/2+q}}\Q
=  \sum_{\substack{-1\le i\le \kappa\\  i\equiv 1 \mod{2}}}  \sum_{q=\kappa-2}^{3\kappa} \frac{\Gamma^{2\mu+4i}(1/4) }{\pi^{3\mu/2+q}}\Q
\end{align*}
and if $\mu=4\ss$ then \eqref{equ:kthDerzn-at1/2} and \eqref{equ:kthDerznz'Dotu-at1/2} imply
\begin{align*}
\frac{d^\kappa}{dy^\kappa}\Big(C'_{4\ss-2,2}(y)\Big)\bigg|_{x=1/2} \in &\,
\sum_{\gl=0}^{[\kappa/2]} \sum_{q=\kappa}^{3\kappa}  \frac{\Gamma^{8\ss+8\gl}(1/4) }{\pi^{6\ss+q}}\Q
+\sum_{\gl=0}^{[(\kappa-1)/2]} \sum_{q=\kappa+1}^{3\kappa+3}  \frac{\Gamma^{8\ss+8\gl}(1/4) }{\pi^{6\ss-3+q}}\Q  \\
\subset &\,  \sum_{\gl=0}^{[\kappa/2]} \sum_{q=\kappa-2}^{3\kappa}  \frac{\Gamma^{8\ss+8\gl}(1/4) }{\pi^{3\mu/2+q}}\Q
= \sum_{\substack{-1\le i\le \kappa\\  i\equiv 0 \mod{2}}} \sum_{q=\kappa-2}^{3\kappa} \frac{\Gamma^{2\mu+4i}(1/4) }{\pi^{3\mu/2+q}}\Q
\end{align*}
Putting these two cases together, we have
\begin{align*}
\int_0^\infty \frac{x^{4p+1} dx}{(\cos x+\cosh x)^m}
\in  &\, \sum_{l=0}^{m/2-1}\sum_{j=0}^{m-1-2l}
\sum_{h=0}^{m/2-1}\sum_{\substack{-1\le i\le \kappa\\ 2\mu+4i\equiv 0 \mod{8}}} \sum_{q=\kappa-2}^{3\kappa}  \frac{\Gamma^{2\mu+4i}(1/4) }{\pi^{3\mu/2+q-4p-2+j}}\Q.
\end{align*}

Note that the power $2\mu+4i$ of $\Gamma$  also satisfies
\begin{equation*}
8p+8-4m\le 2\mu+4i \le 2\mu+4\kappa\le 8p+2m+4+4h\le 8p+4m,
\end{equation*}
while the power of $\pi^{-1}$ satisfies
\begin{align*}
\frac32\mu +q-4p-2+j=&\,\frac32(4p+4+2l-m-2h) +q-4p-2+j  \\
\geq &\, 2p+2+3l-\frac32m-3h+\kappa+j\\
\geq &\, 2p+2+3l-\frac32m-3h+(m-1-2l+2h)\\
\geq &\, 2p-\frac12m-h+1 \\
\geq &\, 2p-m+2
\end{align*}
and
\begin{align*}
\frac32\mu +q-4p-2+j=&\, 2p+4+3l-\frac32m-3h+3(m-1-2l-j+2h)+j\\
\le &\, 2p+1+\frac32m+\frac32m-3=2p+3m-2.
\end{align*}

We have now completed the proof of the theorem.
\end{proof}

\begin{re}
(i) The bounds on the degrees of $\pi^{-1}$ and $\Gamma^4(1/4)$ in Theorem~\ref{thm:Plus} are optimal by numerical evidence.
See \cite[Example 4.2]{RXZ2023} for $m=3$ and \cite[Example 8.13]{XuZhao-2022} for $m=2$.

(ii) Note that the condition $p\geq [m/2]$ is not optimal for the integral in Theorem~\ref{thm:Plus} to converge.
For example, for $(p,m)=(2,6)$ we have the following numerically verified conjecture.
\end{re}

\begin{conj}
Set $\Gamma=\Gamma(1/4)$. Then we have
\begin{align*}
\int_0^\infty \frac{x^9 dx}{(\cos x+\cosh x)^6}
=&\,\frac{-63}{5\cdot 2^{10}}
+ \frac{1071\Gamma^{8}}{5^2\cdot 2^{13}\pi^{2}}
-\frac{21\Gamma^{8}}{2^{12}\pi^{3}}
+\frac{63\Gamma^{8}}{2^{16}\pi^{4}}
-\frac{21\Gamma^{16}}{5^3\cdot 2^{13}\pi^{4}}
+ \frac{3\Gamma^{16}}{5\cdot 2^{13}\pi^{5}}\\
&\,-\frac{161\Gamma^{16}}{5\cdot 2^{19}\pi^{6}}
+\frac{21\Gamma^{16}}{2^{21}\pi^{8}}
+\frac{\Gamma^{24}}{5^2\cdot 2^{19}\pi^{8}}
- \frac{\Gamma^{24}}{3\cdot 2^{20}\pi^{9}}
+ \frac{69\Gamma^{24}}{5\cdot 2^{25}\pi^{10}}\\
&\,-\frac{21\Gamma^{24}}{5\cdot 2^{24}\pi^{11}}
+\frac{63\Gamma^{24}}{5\cdot 2^{27}\pi^{12 }}
-\frac{17\Gamma^{32}}{3\cdot 5^2\cdot 2^{31}\pi^{14}}
+\frac{13\Gamma^{32}}{5\cdot 2^{34}\pi^{16}}
+\frac{3\Gamma^{40}}{5^2\cdot 2^{40}\pi^{20}}.
\end{align*}
\end{conj}

\begin{thm} \label{thm:Minus}
(i) For all integers $p\geq k+1\geq 1$, the Berndt-type integrals
\begin{equation*}
\int_0^\infty \frac{x^{4p-1} dx}{(\cos x-\cosh x)^{2k+1}} \in \Q\Big[\Gamma^4(1/4),\pi^{-1}\Big],
\end{equation*}
where the degrees of $\Gamma^4(1/4)$ are always even and are between $2p-2k$ and $2p+2k$, inclusive,
while the degrees of $\pi^{-1}$ are between $2p-2k$ and $2p+6k$, inclusive.

(ii) For all integers $p\geq k\geq 1$, the Berndt-type integrals
\begin{equation*}
\int_0^\infty \frac{x^{4p+1} dx}{(\cos x-\cosh x)^{2k}} \in \Q\Big[\Gamma^4(1/4),\pi^{-1}\Big],
\end{equation*}
where the degrees of $\Gamma^4(1/4)$  are always even and are between $2p+2-2k$ and $2p+2k$, inclusive, while
the degrees of $\pi^{-1}$ are between $2p-2k+2$ and $2p+6k-2$, inclusive.
\end{thm}

\begin{proof}
From Theorem \ref{main-thm-one-minus}
\begin{align}\label{equ:one-minus}
\int_0^\infty \frac{x^{a} dx}{(\cos -\cosh x)^m}
\in \sum_{l=0}^{[(m-1)/2]}\sum_{j=0}^{m-1-2l}  \Q \pi^{a+1-j} \frac{d^{m-1-2l-j}}{dy^{m-1-2l-j}}
\left\{\sum_{n=1}^\infty (-1)^{mn} \frac{n^{a+1+2l-m}}{\sinh^m (n y)}\right\} \bigg|_{x=1/2}.
\end{align}
If $m=2k+1$ is odd then for all integers $p \geq (m+1)/2$
\begin{align*}
\int_0^\infty \frac{x^{4p-1} dx}{(\cos x-\cosh x)^m}
\in &\,\sum_{l=0}^{(m-1)/2}\sum_{j=0}^{m-1-2l}  \Q \pi^{4p-j}\frac{d^{m-1-2l-j}}{dy^{m-1-2l-j}}
{\bar S}_{4p+2l-m,m}(y) \bigg|_{x=1/2} \\
=&\, \sum_{l=0}^{(m-1)/2}\sum_{j=0}^{m-1-2l}
\sum_{h=0}^{(m-1)/2}\Q \pi^{4p-j}\Big({\bar S}_{4p+2l-m-2h}(y)\Big)^{(m-1-2l-j+2h)}  \bigg|_{x=1/2}
\end{align*}
by Theorem \ref{thm-alter-Sinh-one}. By Lemma \ref{lem-one-exform-sinh-Qxz},
setting $\mu=4p+1+2l-m-2h$ (which is even and $\ge 4$) we get
\begin{align*}		
{\bar S}_{\mu-1,2}(y)=\left\{
                             \begin{array}{ll}
                              {\bar S}_{4\ss+1,1}(y)\in z^{4\ss+2}\Q[v]v', \quad & \quad \hbox{if $\mu=4\ss+2$;} \\
                              {\bar S}_{4\ss-1,1}(y)\in z^{4\ss}\Q[v], \quad & \quad \hbox{if $\mu=4\ss$.}
                             \end{array}
                           \right.
\end{align*}
Set  $\kappa=m-1-2l-j+2h$. If $\mu=4\ss+2$ then \eqref{equ:kthDerznDotu-at1/2} yields
\begin{align*}
\frac{d^\kappa}{dy^\kappa}\Big( {\bar S}_{4\ss+1,1}(y)\Big)\bigg|_{x=1/2} \in &\,
\sum_{\gl=0}^{[(\kappa-1)/2]} \sum_{q=\kappa}^{3\kappa}  \frac{\Gamma^{8\ss+8+8\gl}(1/4) }{\pi^{3\mu/2+q}}\Q
=  \sum_{\substack{0\le i\le \kappa\\  i\equiv 1 \mod{2}}}  \sum_{q=\kappa}^{3\kappa} \frac{\Gamma^{2\mu+4i}(1/4) }{\pi^{3\mu/2+q}}\Q
\end{align*}
and if $\mu=4\ss$ then \eqref{equ:kthDerzn-at1/2} implies
\begin{align*}
\frac{d^\kappa}{dy^\kappa}\Big( {\bar S}_{4\ss-1,1}(y)\Big)\bigg|_{x=1/2} \in &\,
\sum_{\gl=0}^{[\kappa/2]} \sum_{q=\kappa}^{3\kappa}  \frac{\Gamma^{8\ss+8\gl}(1/4) }{\pi^{3\mu/2+q}}\Q
= \sum_{\substack{0\le i\le \kappa\\  i\equiv 0 \mod{2}}} \sum_{q=\kappa-2}^{3\kappa} \frac{\Gamma^{2\mu+4i}(1/4) }{\pi^{3\mu/2+q}}\Q.
\end{align*}
Putting these two cases together, we have
\begin{align*}
\int_0^\infty \frac{x^{4p-1} dx}{(\cos x-\cosh x)^m}
\in &\, \sum_{l=0}^{(m-1)/2}\sum_{j=0}^{m-1-2l}
\sum_{h=0}^{(m-1)/2}\sum_{\substack{0\le i\le \kappa\\ 2\mu+4i\equiv 0 \mod{8}}}\sum_{q=\kappa}^{3\kappa}  \frac{\Gamma^{2\mu+4i}(1/4) }{\pi^{3\mu/2+q-4p+j}}\Q.
\end{align*}
Note that the power $2\mu+4i$ of $\Gamma$ also satisfies
\begin{equation*}
8p-4m+4\le 2\mu+4i \le 2\mu+4\kappa\le 8p+2m-2+4h\le 8p+4m-4,
\end{equation*}
while the power of $\pi$ on the denominator satisfies
\begin{align*}
\frac32\mu +q-4p+j=&\,\frac32(4p+1+2l-m-2h) +q-4p+j  \\
=&\, 2p+\frac12+3l-\frac32m-3h+\kappa+j\\
=&\, 2p-\frac12+3l-\frac32m-3h+(m-2l+2h)\\
\geq &\, 2p-\frac32-\frac12m-\frac12(m-1)=2p-m-1=2p-2k
\end{align*}
and
\begin{align*}
\frac32\mu +q-4p+j\le &\, 2p+\frac32+3l-\frac32m-3h+3(m-1-2l-j+2h)+j\\
\le &\, 2p-\frac32+\frac32m+\frac32(m-1)=2p+3m-3=2p+6k.
\end{align*}

Similarly, if $m=2k$ is even then for all integers $p\geq m/2$ by \eqref{equ:one-minus}
\begin{align*}
\int_0^\infty \frac{x^{4p+1} dx}{(\cos x-\cosh x)^m}
\in &\,\sum_{l=0}^{m/2-1}\sum_{j=0}^{m-1-2l}  \Q \pi^{4p+2-j}\frac{d^{m-1-2l-j}}{dy^{m-1-2l-j}}
S_{4p+2+2l-m,m}(y) \bigg|_{x=1/2} \\
=&\, \sum_{l=0}^{m/2-1}\sum_{j=0}^{m-1-2l}
\sum_{h=0}^{m/2-1}\Q \pi^{4p+2-j}\Big(S_{4p+2+2l-m-2h,2}(y)\Big)^{(m-1-2l-j+2h)}  \bigg|_{x=1/2}
\end{align*}
by Theorem \ref{thm-noalter-Sinh-one}. By Lemma \ref{lem-one-noexform-sinh-Qxz},
setting $\mu=4p+4+2l-m-2h$ (which is even and $\ge 6$) we get
\begin{equation*}		
S_{\mu-2,2}(y)=\left\{
                \begin{array}{ll}
                      S_{4\ss,2}(y)   \in z^{4\ss+2}\Q[v]v'+z^{4\ss+1}z'\Q[v], \quad & \quad \hbox{if $\mu=4\ss+2\geq 6$;} \\
                      S_{4\ss-2,2}(y)   \in z^{4\ss}\Q[v]+z^{4\ss-1}z'\Q[v]v', \quad & \quad \hbox{if $\mu=4\ss\geq 8$.}
                \end{array}
                \right.
\end{equation*}
Then the rest of the proof goes verbatim as that for the even $m$ case in Theorem~\ref{thm:Plus}.
Therefore
\begin{align*}
\int_0^\infty \frac{x^{4p+1} dx}{(\cos x-\cosh x)^m}
\in  &\, \sum_{l=0}^{m/2-1}\sum_{j=0}^{m-1-2l}
\sum_{h=0}^{m/2-1}\sum_{\substack{-1\le i\le \kappa\\ 2\mu+4i\equiv 0 \mod{8}}} \sum_{q=\kappa-2}^{3\kappa}  \frac{\Gamma^{2\mu+4i}(1/4) }{\pi^{3\mu/2-2+q-4p+j}}\Q
\end{align*}
and the bounds for powers of $\Gamma(1/4)$ and $\pi^{-1}$ can be computation in exactly the same way as well.
This concludes the proof of the theorem.
\end{proof}

\begin{re}
The bounds on the degrees of $\pi^{-1}$ and $\Gamma^4(1/4)$ in Theorem~\ref{thm:Minus} are optimal by numerical evidence.
See \cite[Example 4.5]{RXZ2023} for $m=3$ and \cite[Example 6.4]{XZ2023} for $m=2$.
\end{re}

We end our paper with two conjectures that are supported by extensive numerical evidence.

\begin{con}
For all integers $n\geq 1$, the Berndt-type integrals
\begin{align*}
\int_0^\infty \frac{x\, dx}{(\cos x+\cosh x)^{2n-1}} & \in
\sum_{j=1}^{n-1} \frac{\Gamma^{8j-4}(1/4)}{\pi^{6j-4}}\Q+\sum_{j=1}^{n} \frac{\Gamma^{8j-4}(1/4)}{\pi^{6j-5}} \Q, \\
\int_0^\infty \frac{x\, dx}{(\cos x+\cosh x)^{2n}}  & \in
\Q+\sum_{j=1}^{n-1} \frac{\Gamma^{8j}(1/4)}{\pi^{6j-1}} \Q+\sum_{j=1}^{n} \frac{\Gamma^{8j}(1/4)}{\pi^{6j-2}} \Q.
\end{align*}
\end{con}

\begin{con} Set $\Gamma=\Gamma(1/4)$. For all integers $n\geq 1$, the Berndt-type integrals
\begin{align*}
\int_0^\infty \frac{x^5\, dx}{(\cos x+\cosh x)^{2n-1}} \in  &
\frac{\Gamma^{4}}{\pi}\Q+\frac{\Gamma^{4}}{\pi^2}\Q+
\sum_{j=2}^{n-3}\sum_{i=1}^{6} \frac{\Gamma^{8j-4}}{\pi^{6j-10+i}}\Q
+\sum_{j=1}^{3}\sum_{i=1}^{7-2j}  \frac{\Gamma^{8j+8n-20}}{\pi^{6j+6n-22+i}}\Q, \\
\int_0^\infty \frac{x^5\, dx}{(\cos x+\cosh x)^{2n}} \in &
\Q+\sum_{i=1}^{4} \frac{\Gamma^{8}}{\pi^{i+1}}\Q+\sum_{i=1,i\ne 2}^{6} \frac{\Gamma^{16}}{\pi^{i+5}}\Q \\
 &+ \sum_{j=3}^{n-2}\sum_{i=1}^{6} \frac{\Gamma^{8j}}{\pi^{6j-7+i}}\Q
+\sum_{j=1}^{3}\sum_{i=1}^{7-2j}  \frac{\Gamma^{8j+8n-16}}{\pi^{6j+6n-19+i}}\Q.
\end{align*}
\end{con}

\medskip
{\bf Acknowledgments.} The authors expresses their deep gratitude to Professors Bruce C. Berndt and Alexey Kuznetsov for valuable discussions and comments. Ce Xu is supported by the National Natural Science Foundation of China (Grant No. 12101008), the Natural Science Foundation of Anhui Province (Grant No. 2108085QA01) and the University Natural Science Research Project of Anhui Province (Grant No. KJ2020A0057). Jianqiang Zhao is supported by the Jacobs Prize from The Bishop's School.

\end{document}